\documentclass[11pt]{article}
\usepackage{hyperref}
\usepackage{titling}
\newcommand{\subtitle}[1]{%
  \posttitle{%
    \par\vspace{-.25cm}\end{center}
    \begin{center}\large#1\end{center}
    \vskip0.5em}%
}

\usepackage{amsmath}
\usepackage{amsthm}
\usepackage{amsfonts}
\usepackage{amssymb}
\usepackage{comment}
\usepackage{hyperref}
\usepackage[margin=1in]{geometry}
\usepackage{graphicx}
\usepackage{epstopdf}
\usepackage{youngtab}

\DeclareMathOperator{\PP}{\mathbb{P}}
\DeclareMathOperator{\CC}{\mathbb{C}}
\DeclareMathOperator{\RR}{\mathbb{R}}

\DeclareMathOperator{\ZZ}{\mathbb{Z}}
\DeclareMathOperator{\Sch}{\mathfrak{S}}

\DeclareMathOperator{\isom}{\cong}

\DeclareMathOperator{\Gr}{Gr}
\DeclareMathOperator{\Fl}{Fl}
\DeclareMathOperator{\GL}{GL}
\DeclareMathOperator{\hook}{hook}
\DeclareMathOperator{\inv}{inv}
\newtheorem{theorem}{Theorem}[section]
\newtheorem{problem}[theorem]{Open Problem}
\newtheorem{lemma}[theorem]{Lemma}
\newtheorem{question}[theorem]{Question}
\newtheorem{corollary}[theorem]{Corollary}

\theoremstyle{definition}

\newtheorem{fact}[theorem]{Fact}
\newtheorem{definition}[theorem]{Definition}
\newtheorem{exercise}[theorem]{Exercise}
\newtheorem{remark}[theorem]{Remark}
\newtheorem{example}[theorem]{Example}

\title{Variations on a Theme of Schubert Calculus}

\author{Maria Gillespie\thanks{Supported by NSF MSPRF grant PDRF 1604262.}  \\ \href{mailto:mgillespie@math.ucdavis.edu}{mgillespie@math.ucdavis.edu}}

\date{\today}
\begin{document}

\maketitle{}

\begin{abstract}
	In this tutorial, we provide an overview of many of the established combinatorial and algebraic tools of Schubert calculus, the modern area of enumerative geometry that encapsulates a wide variety of topics involving intersections of linear spaces.  It is intended as a guide for readers with a combinatorial bent to understand and appreciate the geometric and topological aspects of Schubert calculus, and conversely for geometric-minded readers to gain familiarity with the relevant combinatorial tools in this area. 
	
	We lead the reader through a tour of three variations on a theme: Grassmannians, flag varieties, and orthogonal Grassmannians.  The orthogonal Grassmannian, unlike the ordinary Grassmannian and the flag variety, has not yet been addressed very often in textbooks, so this presentation may be helpful as an introduction to type B Schubert calculus.
	
	This work is adapted from the author's lecture notes for a graduate workshop during the Equivariant Combinatorics Workshop at the Center for Mathematics Research, Montreal, June 12-16, 2017.
\end{abstract}

\section{Introduction}\label{sec:intro}

Schubert calculus was invented as a general method for solving linear intersection problems in Euclidean space.  One very simple example of a linear intersection problem is the following:  How many lines pass through two given points in the plane?  

It is almost axiomatically true that the answer is $1$, as long as the points are distinct (otherwise it is $\infty$).  Likewise, we can ask how many points are contained in two lines in the plane.  The answer is also usually $1$, though it can be $0$ if the lines are parallel, or $\infty$ if the lines are equal.  

In higher dimensions the answers may change: in three-dimensional space, there are most often zero points of intersection of two given lines.  One can also consider more complicated intersection problems involving subspaces of Euclidean space.  For instance, how many planes in $4$-space contain a given line and a given point?  Generically, the answer will be $1$, but in degenerate cases (when the point is on the line) may be $\infty$.

It seems that the answers to such problems are often $1$, $0$, or $\infty$, but this is not always the case.  Here is the classic example of Schubert calculus, where the answer is generically $2$:

\begin{question}\label{lines}
	How many lines intersect four given lines in three-dimensional space?
\end{question}

Hermann Schubert's 19th century solution to this question\footnote{See \cite{Schubert} for Schubert's original work, or \cite{Ronga} for a modern exposition on Schubert's methods.} would have invoked what he called the ``Principle of Conservation of Number'' as follows.  Suppose the four lines $l_1,l_2,l_3,l_4$ were arranged so that $l_1$ and $l_2$ intersect at a point $P$, $l_2$ and $l_3$ intersect at $Q$, and none of the other pairs of lines intersect and the planes $\rho_1$ and $\rho_2$ determined by $l_1,l_2$ and $l_3,l_4$ respectively are not parallel.  Then $\rho_1$ and $\rho_2$ intersect at another line $\alpha$, which necessarily intersects all four lines.  The line $\beta$  through $P$ and $Q$ also intersects all four lines, and it is not hard to see that these are the only two in this case.  

Schubert would have said that since there are two solutions in this configuration, there are two for every configuration of lines for which the number of solutions is finite, since the solutions can be interpreted as solutions to polynomial equations over the complex numbers.   The answer is indeed preserved in this case, but the lack of rigor in this method regarding multiplicities led to some errors in computations in harder questions of enumerative geometry. 

The following is an example of a more complicated enumerative geometry problem, which is less approachable with elementary methods.

\begin{question}\label{complicated} 
	How many $k$-dimensional subspaces of $\mathbb{C}^{n}$ intersect each of $k\cdot (n-k)$ fixed subspaces of dimension $n-k$ nontrivially?  
\end{question}

Hilbert's 15th problem asked to put Schubert's enumerative methods on a rigorous foundation.  This led to the modern-day theory known as Schubert calculus.

The main idea, going back to Question \ref{lines}, is to let $X_i$ be the space of all lines $L$ intersecting $l_i$ for each $i=1,\ldots,4$.  Then the intersection $X_1\cap X_2\cap X_3\cap X_4$ is the set of solutions to our problem.  Each $X_i$ is an example of a \textit{Schubert variety}, an algebraic and geometric object that is essential to solving these types of intersection problems.

\subsection{`Variations on a Theme'}

This tutorial on Schubert calculus is organized as a theme and variations\footnote{A play on words that references the shared surname with musical composer Franz Schubert, who also lived in Germany in the 19th century.}.  In particular, after briefly recalling some of the necessary geometric background on projective spaces in Section \ref{sec:background} (which may be skipped or skimmed over by readers already familiar with these basics), we begin in Section \ref{sec:theme} (the `Theme') with the foundational ideas of Schubert calculus going back to Schubert \cite{Schubert}.  This includes a rigorous development of Schubert varieties in the \textit{Grassmannian}, the set of all $k$-dimensional subspaces of a fixed $n$-dimensional space, and a more careful geometric analysis of the elementary intersection problems mentioned in the introduction.  We also establish the basic properties of the Grassmannian.  Much of this background material can also be found in expository sources such as \cite{Fulton},  \cite{KleimanLaksov}, and \cite{Manivel}, and much of the material in the first few sections is drawn from these works.

In Variation 1 (Section \ref{sec:variation1}), we present the general formulas for intersecting complex Schubert varieties, and show how it relates to calculations in the cohomology of the Grassmannian as well as products of Schur functions.  Variation 2 (Section \ref{sec:variation2}) repeats this story for the complete flag variety (in place of the Grassmannian), with the role of Schur functions replaced by the Schubert polynomials.  Finally, Variation 3 (Section \ref{sec:variation3}) explores Schubert calculus in the ``Lie type B'' Grassmannian, known as the \textit{orthogonal Grassmannian}.

There are countless more known variations on the theme of classical Schubert calculus, including Grassmannians in the remaining Lie types, partial flag varieties,  and Schubert varieties over the real numbers.  There is also much that has yet to be explored.  We conclude with an overview of some of these potential further directions of study in Section \ref{sec:conclusion}.

\subsection{Acknowledgments}

The author thanks Jennifer Morse, Fran\c{c}ois Bergeron, Franco Saliola, and Luc Lapointe for the invitation to teach a graduate workshop on Schubert calculus at the Center for Mathematics Research in Montreal, for which these extended notes were written.  Thanks also to Sara Billey and the anonymous reviewer for their extensive feedback.  Thanks to Helene Barcelo, Sean Griffin, Philippe Nadeau, Alex Woo, Jake Levinson, and Guanyu Li for further suggestions and comments.   Finally, thanks to all of the participants at the graduate workshop for their comments, questions, and corrections that greatly improved this exposition.

\section{Background on Projective Space}\label{sec:background}

  The notion of \textit{projective space} helps clean up many of the ambiguities in the question above.  For instance, in the projective plane, parallel lines meet, at a ``point at infinity''.\footnote{Photo of the train tracks downloaded from edupic.net.}   It also is one of the simplest examples of a Schubert variety.
  
  \begin{figure}[h]
  	\begin{center}
  	\includegraphics[width=200pt]{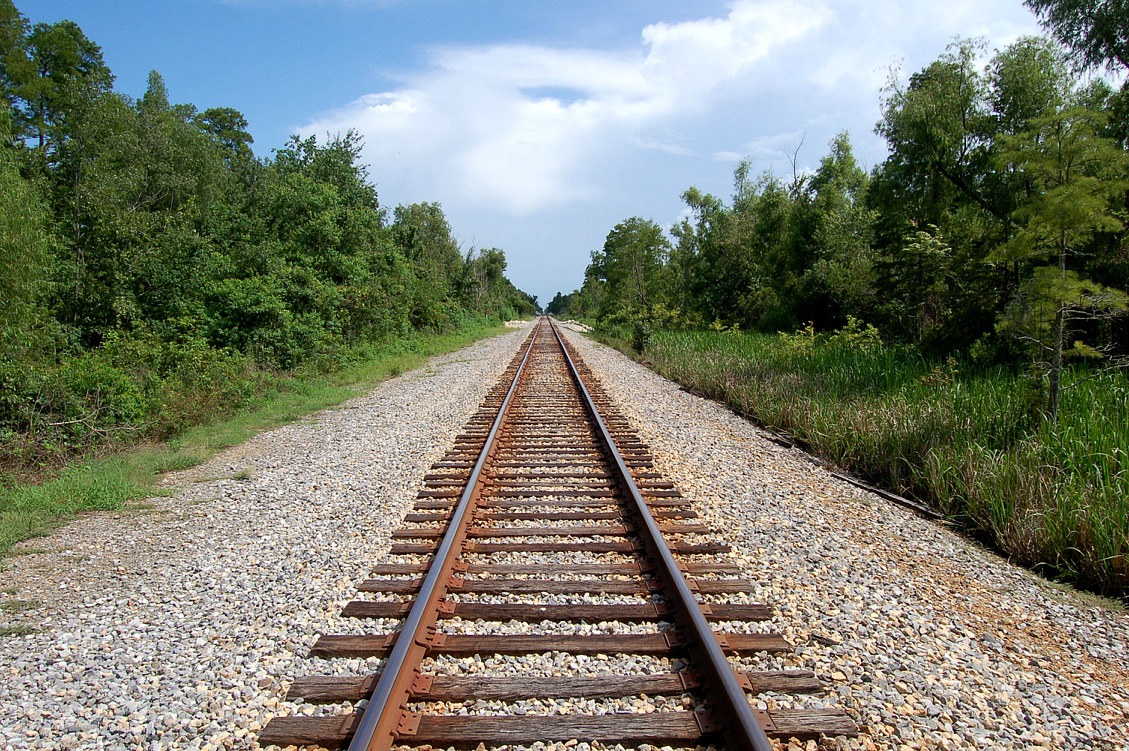}
  	\caption{Parallel lines meeting at a point at infinity.}
    \end{center}
  \end{figure}
  
  One way to define projective space over a field $k$ is as the set of lines through the origin in one higher dimensional space as follows.
  
  \begin{definition}
  	The $n$-dimensional \textbf{projective space} $\PP^n_k$ over a field $k$ is the set of equivalence classes in $k^{n+1}\setminus\{(0,0,\ldots,0)\}$ with respect to the relation $\sim$ given by scalar multiplication, that is, $$(x_0,x_1,\ldots,x_n)\sim (y_0,y_1,\ldots,y_n)$$ if and only if there exists $a\in k\setminus \{0\}$ such that $ax_i=y_i$ for all $i$.  We write $(x_0:x_1:\cdots:x_n)$ for the equivalence class in $\PP^k$ containing $(x_0,\ldots,x_n)$, and we refer to $(x_0:x_1:\cdots:x_n)$ as a \textit{point} in $\PP^k$ in \textit{homogeneous coordinates}.
  \end{definition}
  
  Note that \textit{a point in $\PP^n_k$ is a line through the origin in $k^{n+1}$}.  In particular, a line through the origin consists of all scalar multiples of a given nonzero vector.
  
  Unless we specify otherwise, we will always use $k=\mathbb{C}$ and simply write $\PP^n$  for $\PP^n_{\CC}$ throughout these notes.
  
  \begin{example}
  	In the ``projective plane'' $\PP^2$, the symbols $(2:0:1)$ and $(4:0:2)$ both refer to the same point.    
  \end{example}
  
  It is useful to think of projective space as having its own geometric structure, rather than just as a quotient of a higher dimensional space.  In particular, a \textbf{geometry} is often defined as a set along with a group of transformations.  A \textbf{projective transformation} is a map $f:\mathbb{P}^n\to \mathbb{P}^n$ of the form $$f(x_0:x_1:\cdots:x_n)=(y_0:y_1:\cdots:y_n)$$ where for each $i$, $$y_i=a_{i0}x_0+a_{i1}x_1+\cdots a_{in}x_n$$ for some fixed constants $a_{ij}\in \mathbb{C}$ such that the $(n+1)\times (n+1)$ matrix $(a_{ij})$ is invertible.  
  
  Notice that projective transformations are well-defined on $\mathbb{P}^n$ because scaling all the $x_i$ variables by a constant $c$ has the effect of scaling the $y$ variables by $c$ as well.  This is due to the fact that the defining equations are \textbf{homogeneous}: every monomial on both sides of the equation has a fixed degree $d$ (in this case $d=1$).    
  
  \subsection{Affine patches and projective varieties}
  
  There is another way of thinking of projective space: as ordinary Euclidean space with extra smaller spaces placed out at infinity.  For instance, in $\PP^1$, any point $(x:y)$ with $y\neq 0$ can be rescaled to the form $(t:1)$.  All such points can be identified with the element $t\in \mathbb{C}$, and then there is only one more point in $\mathbb{P}^1$, namely $(1:0)$.  We can think of $(1:0)$ as a point ``at infinity'' that closes up the \textit{affine line} $\CC^1$ into the ``circle'' $\PP^1$.  Thought of as a real surface, the complex $\PP^1$ is actually a sphere.
  
  Similarly, we can instead parameterize the points $(1:t)$ by $t\in \CC^1$ and have $(0:1)$ be the extra point.  The subsets given by $\{(1:t)\}$ and $\{(t:1)\}$ are both called \textbf{affine patches} of $\PP^1$, and form a cover of $\PP^1$, from which we can inherit a natural topology on $\PP^1$ from the Euclidean topology on each $\CC^1$.  In fact, the two affine patches form an open cover in this topology, so $\PP^1$ is compact.
  
  As another example, the projective plane $\PP^2$ can be written as the disjoint union $$\{(x:y:1)\}\sqcup \{(x:1:0)\}\sqcup \{1:0:0\}=\CC^2\sqcup \CC^1\sqcup \CC^0,$$ which we can think of as a certain closure of the affine patch $\{(x:y:1)\}$.  The other affine patches are $\{(x:1:y)\}$ and $\{(1:x:y)\}$ in this case.
  
  We can naturally generalize this as follows.
  
  \begin{definition}
   The \textbf{standard affine patches} of $\PP^n$ are the sets $$\{(t_0:t_1:\cdots:t_{i-1}:1:t_{i+1}:\cdots:t_n)\}\isom \CC^n$$ for $i=0,\ldots,n$.  
  \end{definition} 
  
  An \textbf{affine variety} is usually defined as the set of solutions to a set of polynomials in $k^n$ for some field $k$.  For instance, the graph of $y=x^2$ is an affine variety in $\mathbb{R}^2$, since it is the set of all points $(x,y)$ for which $f(x,y)=y-x^2$ is zero. 
        \begin{center}
          \includegraphics{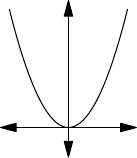}
        \end{center}
  
   In three-dimensional space, we might consider the plane defined by the \textit{zero locus} of $f(x,y,z)=x+y+z$, that is, the set of solutions to $f(x,y,z)=0$.  Another example is the line $x=y=z$ defined by the common zero locus of $f(x,y,z)=x-y$ and $g(x,y,z)=x-z$.
  
  Recall that a polynomial is \textbf{homogeneous} if all of its terms have the same total degree.  For instance, $x^2+3yz$ is homogeneous because both terms have degree $2$, but $x^2-y+1$ is not homogeneous.
  
  \begin{definition}
  	A \textbf{projective variety} is the common zero locus in $\PP^n$ of a finite set of homogeneous polynomials $f_1(x_0,\ldots,x_n),\ldots, f_r(x_0,\ldots,x_n)$ in $\PP^n$.  We call this variety $V(f_1,\ldots,f_r)$.  In other words, $$V(f_1,\ldots,f_r)=\{(a_0:\cdots :a_n)\mid f_i(a_0:\cdots:a_n)=0\text{ for all }i \}.$$
  \end{definition}
  
  \begin{remark}
  	  Note that we need the homogeneous condition in order for projective varieties to be well-defined.  For instance, if $f(x,y)=y-x^2$ then $f(2,4)=0$ and $f(4,8)\neq 0$, but $(2:4)=(4:8)$ in $\PP^1$.  So the value of a nonhomogeneous polynomial on a point in projective space is not, in general, well-defined.
  \end{remark}
  
  The intersection of a projective variety with the $i$-th affine patch is the \textit{affine} variety formed by setting $x_i=1$ in all of the defining equations.  For instance, the projective variety in $\PP^2$ defined by $f(x:y:z)=yz-x^2$ restricts to the affine variety defined by $f(x,y)=y-x^2$ in the affine patch $z=1$.  
  
  We can also reverse this process.  The \textbf{homogenization} of a polynomial $f(x_0,\ldots, x_{n-1})$ in $n$ variables using another variable $x_n$ is the unique homogeneous polynomial $g(x_0:\cdots:x_{n-1}:x_n)$ with $\deg(g)=\deg(f)$ for which $$g(x_0:\cdots:x_{n-1}:1)=f(x_0,\ldots,x_{n-1}).$$
  For instance, the homogenization of $y-x^2$ is $yz-x^2$.  If we homogenize the equations of an affine variety, we get a projective variety which we call its \textbf{projective closure}.
  
  \begin{example}
  	 The projective closure of the parabola defined by $y-x^2-1=0$ is the projective variety in $\PP^3$ defined by the equation $yz-x^2-z^2=0$.  If we intersect this with the $y=1$ affine patch, we obtain the affine variety $z-x^2-z^2=0$ in the $x,z$ variables.  This is the circle $x^2+(z-\frac{1}{2})^2=\frac{1}{4}$, and so parabolas and circles are essentially the same object in projective space, cut different ways into affine patches.  
  	 
  	 As explained in more detail in Problem \ref{prob:conics} below, there is only one type of (nondegenerate) conic in projective space.
  \end{example}
  
  \begin{remark}
  	The above example implies that if we draw a parabola on a large, flat plane and stand at its apex, looking out to the horizon we will see the two branches of the parabola meeting at a point on the horizon, closing up the curve into an ellipse.\footnote{Unfortunately, we could not find any photographs of parabolic train tracks.}
  \end{remark}

  \subsection{Points, lines, and $m$-planes in projective space}
  
   Just as the points of $\PP^n$ are the images of lines in $\CC^{n+1}$, a \textbf{line} in projective space can be defined as the image of a \textbf{plane} in $k^{n+1}$, and so on.  We can define these in terms of homogeneous coordinates as follows.
  
  \begin{definition}
  	An \textbf{$(n-1)$-plane} or \textbf{hyperplane} in $\PP^n$ is the set of solutions $(x_0:\cdots: x_n)$ to a homogeneous linear equation $$a_0x_0+a_1x_1+\cdots +a_nx_n=0.$$  A \textbf{$k$-plane} is an intersection of $n-k$ hyperplanes, say $a_{i0}x_0+a_{i1}x_1+\cdots +a_{in}x_n=0$ for $i=1,\ldots,n-k$, such that the matrix of coefficients $(a_{ij})$ is full rank.
  \end{definition}
  
  \begin{example}
  	In the projective plane $\PP^n$, the line $l_1$ given by $2x+3y+z=0$ restricts to the line $2x+3y+1=0$ in the affine patch $z=1$.  Notice that the line $l_2$ given by $2x+3y+2z=0$ restricts to $2x+3y+2=0$ in this affine patch, and is parallel to the restriction of $l_1$ in this patch.  However, the projective closures of these affine lines intersect at the point $(3:-2:0)$, on the $z=0$ line at infinity.
  	
  	In fact, any two distinct lines meet in a point in the projective plane.  In general, intersection problems are much easier in projective space.  See Problem \ref{V2} below to apply this to our problems in Schubert Calculus. 
  \end{example}

  \subsection{Problems}
  
  \begin{enumerate}
  	\item \textbf{Transformations of $\PP^1$:} Show that a projective transformation on $\PP^1$ is uniquely determined by where it sends $0=(0:1)$, $1=(1:1)$, and $\infty=(1:0)$. 
  	
  	\item \textbf{Choice of $n+2$ points stabilizes $\PP^n$:} Construct a set $S$ of $n+2$ distinct points in $\PP^n$ for which any projective transformation is uniquely determined by where it sends each point of $S$.  What are necessary and sufficient conditions for a set of $n+2$ distinct points in $\PP^n$ to have this property?
  	
  	\item\label{prob:conics} \textbf{All conics in $\PP^2$ are the same:} Show that, for any quadratic homogeneous polynomial $f(x,y,z)$ there is a projective transformation that sends it to one of $x^2$, $x^2+y^2$, or $x^2+y^2+z^2$.  Conclude that any two ``nondegenerate'' conics are the same up to a projective transformation.  
  	
  	(Hint: Any quadratic form can be written as $\mathbf{x}A\mathbf{x}^T$ where $\mathbf{x}=(x,y,z)$ is the row vector of variables and $\mathbf{x}^T$ is its transpose, and $A$ is a symmetric matrix, with $A=A^T$.  It can be shown that a symmetric matrix $A$ can be diagonalized, i.e., written as $BDB^T$ for some diagonal matrix $D$.  Use the matrix $B$ as a projective transformation to write the quadratic form as a sum of squares.) 
  
    \item \label{V2} \textbf{Schubert Calculus in Projective Space:} The question of how many points are contained in two distinct lines in $\CC^2$ can be ``projectivized'' as follows: if we ask instead how many points are contained in two distinct lines in $\PP^2$, then the answer is always $1$ since parallel lines now intersect, a much nicer answer!  
    
    Write out projective versions of Questions \ref{lines} and \ref{complicated}.  What do they translate to in terms of intersections of subspaces of one-higher-dimensional affine space?
  \end{enumerate}

\pagebreak

%%%%%% DAY 2 %%%%%%%%%%%%%%%
  
\section{Theme: The Grassmannian}\label{sec:theme}

  Not only does taking the projective closure of our problems in $\PP^n$ make things easier, it is also useful to think of the intersection problems as involving subspaces of $\CC^{n+1}$ rather than $k$-planes in $\PP^n$.  The definition of the Grassmannian below is analogous to our first definition of projective space.

  \begin{definition}
  	The \textbf{Grassmannian} $\Gr(n,k)$ is the set of all $k$-dimensional subspaces of $\CC^n$.  
  \end{definition}
  
  As in projective spaces, we call the elements of $\Gr(n,k)$ the ``points'' of $\Gr(n,k)$, even though they are defined as entire subspaces of $\CC^n$.  We will see soon that $\Gr(n,k)$ has the structure of a projective variety, making this notation useful.
  
  Every point of the Grassmannian can be described as the span of some $k$ independent row vectors of length $n$, which we can arrange in a $k\times n$ matrix.  For instance, the matrix 
  $$\left[\begin{array}{ccccccc}
  0 & -1 & -3 & -1 & 6 & -4 & 5 \\
  0 & 1 & 3 & 2 & -7 & 6 & -5 \\
  0 & 0 & 0 & 2 & -2 & 4 & -2 
  \end{array}\right]$$
  represents a point in $\Gr(7,3)$.  Notice that we can perform elementary row operations on the matrix without changing the point of the Grassmannian it represents.  We will use the convention that the pivots will be in order from left to right and bottom to top. 
  
  \begin{exercise}
  	Show that the matrix above has reduced row echelon form: 
  	$$\left[\begin{array}{ccccccc}
  	0 & 0 & 0 & 0 & 0 & 0 & 1 \\
  	0 & 0 & 0 & 1 & \ast & \ast & 0 \\
  	0 & 1 & \ast & 0 & \ast & \ast & 0 \\
  	\end{array}\right],$$ where the $\ast$ entries are certain complex numbers.  
  	
  \end{exercise}
  
  	We can summarize our findings as follows.
  	
  	\begin{fact}
  		Each point of $\Gr(n,k)$ is the row span of a unique full-rank $k\times n$ matrix in reduced row echelon form.
  	\end{fact}  
  	
  	The subset of the Grassmannian whose points have a particular reduced row echelon form constitutes a \textbf{Schubert cell.}  Notice that $\Gr(n,k)$ is a disjoint union of Schubert cells.

\subsection{Projective variety structure}
  
The Grassmannian can be viewed as a projective variety by embedding $\Gr(n,k)$ in $\PP^{\binom{n}{k}-1}$ via the \textit{Pl\"ucker embedding}.  To do so, choose an ordering on the $k$-element subsets $S$ of $\{1,2,\ldots,n\}$ and use this ordering to label the homogeneous coordinates $x_S$ of $\PP^{\binom{n}{k}-1}$.  Now, given a point in the Grassmannian represented by a matrix $M$, let $x_S$ be the determinant of the $k\times k$ submatrix determined by the columns in the subset $S$.  This determines a point in projective space since row operations can only change the determinants up to a constant factor, and the coordinates cannot all be zero since the matrix has rank $k$.  

For example, in $\Gr(4,2)$, the matrix $$\left[\begin{array}{cccc}
0 & 0 & 1 & 2 \\
1 & -3 & 0 & 3
\end{array} \right]$$
has Pl\"ucker coordinates given by the determinants of all the $2\times 2$ submatrices formed by choosing two of the columns above.  We write $x_{ij}$ for the determinant formed columns $i$ and $j$, so for instance, $x_{24}=\det\left(\begin{array}{cc}
0 & 2 \\
-3 & 3
\end{array}\right)=6$.  If we order the coordinates $(x_{12}:x_{13}:x_{14}:x_{23}:x_{24}:x_{34})$ then the image of the above point under the Pl\"ucker embedding is is $(0:-1:-2:3:6:3)$.

One can show that the image is a projective variety in $\PP^{\binom{n}{k}-1}$, cut out by homogeneous quadratic relations in the variables $x_S$ known as the \textit{Pl\"ucker relations}.  See \cite{CoxLittleOShea}, pg.\ 408 for details.

  \subsection{Schubert cells and Schubert varieties}
 
 To enumerate the Schubert cells in the Grassmannian, we assign to the matrices of the form 
 $$\left[\begin{array}{ccccccc}
 0 & 0 & 0 & 0 & 0 & 0 & 1 \\
 0 & 0 & 0 & 1 & \ast & \ast & 0 \\
 0 & 1 & \ast & 0 & \ast & \ast & 0 \\
 \end{array}\right]$$
 a \textbf{partition}, that is, a nonincreasing sequence of nonnegative integers $\lambda=(\lambda_1,\ldots,\lambda_k)$, as follows.  Cut out the $k\times k$ staircase from the upper left corner of the matrix, and let $\lambda_i$ be the distance from the edge of the staircase to the $1$ in row $i$.  In the example shown, we get the partition $\lambda=(4,2,1)$.  Notice that we always have $\lambda_1\ge \lambda_2\ge \cdots \ge \lambda_k$.
 \begin{center}
 	\includegraphics{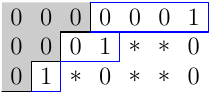}
 \end{center}
 
 \begin{definition}
 	The \textbf{size} of a partition $\lambda$, denoted $|\lambda|$, is $\sum_i \lambda_i$, and its \textbf{length}, denoted $l(\lambda)$, is the number of nonzero parts.  The entries $\lambda_i$ are called its \textbf{parts}.
 \end{definition}	
 
 \begin{remark} 		
 	With this notation, Schubert cells in $\Gr(n,k)$ are in bijection with the partitions $\lambda$ for which $l(\lambda)\le k$ and $\lambda_1\le n-k$.
 \end{remark}
 
 \begin{definition}
 	The \textbf{Young diagram} of a partition $\lambda$ is the left-aligned partial grid of boxes in which the $i$-th row from the top has $\lambda_i$ boxes.  
 \end{definition}
 
 For example, the Young diagram of the partition $(4,2,1)$ that came up in the previous example is shown as the shaded boxes in the diagram below. By identifying the partition with its Young diagram, we can alternatively define $\lambda$ as the complement in a $k\times (n-k)$ rectangle of the diagram $\mu$ defined by the right-aligned shift of the $\ast$ entries in the matrix:
 
 \begin{center}
 	\includegraphics{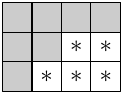}
 \end{center}
 
 Since the $k\times (n-k)$ rectangle is the bounding shape of our allowable partitions, we will call it the \textbf{ambient rectangle}.
 
 \begin{definition}
 	For a partition $\lambda$ contained in the ambient rectangle, the \textbf{Schubert cell} $\Omega_{\lambda}^\circ$ is the set of points of $\Gr(n,k)$ whose row echelon matrix has corresponding partition $\lambda$.  Explicitly, $$\Omega^\circ_{\lambda}=\{V\in \Gr(n,k)\mid \dim(V\cap \langle e_1,\ldots, e_{r}\rangle) = i\text{ for }n-k+i-\lambda_i\le r \le n-k+i-\lambda_{i+1}\}.$$  Here $e_{n-i+1}$ is the $i$-th standard unit vector $(0,0,\ldots,0,1,0,\ldots,0)$ with the $1$ in the $i$-th position, so $e_1=(0,0,\ldots,1)$, $e_2=(0,0,\ldots,1,0)$, and so on.  The notation $\langle e_1,\ldots,e_r\rangle$ denotes the span of the vectors $e_1,\ldots,e_r$.
 \end{definition}

 \begin{remark}
 	The numbers $n-k+i-\lambda_i$ are the positions of the $1$'s in the matrix counted from the right.
 \end{remark}

 Since each $\ast$ can be any complex number, we have $\Omega_{\lambda}^\circ= \CC^{k(n-k)-|\lambda|}$ as a set, and so $$\dim(\Omega_{\lambda}^\circ)=k(n-k)-|\lambda|.$$  In particular the dimension of the Grassmannian is $k(n-k)$.
 
 We are now in a position to define \textbf{Schubert varieties} as closed subvarieties of the Grassmannian.  
 
 \begin{definition}
 	The \textbf{standard Schubert variety} corresponding to a partition $\lambda$, denoted $\Omega_\lambda$, is the set $$\Omega_{\lambda}=\{V\in \Gr(n,k)\mid \dim(V\cap \langle e_1,\ldots, e_{n-k+i-\lambda_i}\rangle) \ge i\}.$$
 \end{definition}
 
 \begin{remark}
 	In the topology on the Grassmannian, as inherited from projective space via the Pl\"ucker embedding, the Schubert variety $\Omega_\lambda$ is the closure $\overline{{\Omega_\lambda}^\circ}$ of the corresponding Schubert cell.  We will explore more of the topology of the Grassmannian in section \ref{sec:variation1}.
 \end{remark}
 
  Note that we have $\dim(\Omega_\lambda)=\dim(\Omega_\lambda^{\circ}) = k(n-k)-|\lambda|$ as well.
 
 \begin{example}
 	Consider the Schubert variety $\Omega_{\tiny\yng(2)}$ in $\PP^5=\Gr(6,1)$.  The ambient rectangle is a $1\times 5$ row of squares.  There is one condition defining the points $V\in \Omega_{\tiny\yng(2)}$, namely $\dim(V\cap \langle e_1,e_2,e_3,e_4\rangle) \ge 1$, where $V$ is a one-dimensional subspace of $\CC^6$.  This means that $V$ is contained in $\langle e_1,\ldots,e_4\rangle $, and so, expressed in homogeneous coordinates, its first two entries (in positions $e_5$ and $e_6$) are $0$.  
 	
 	Thus each point of $\Omega_{\tiny\yng(2)}$ can be written in one of the following forms: \begin{center} $(0:0:1:\ast:\ast:\ast)$ \\ $(0:0:0:1:\ast:\ast)$ \\ $(0:0:0:0:1:\ast)$ \\$(0:0:0:0:0:1)$ \end{center}  It follows that $\Omega_{\tiny\yng(2)}$ can be written as a disjoint union of Schubert cells as follows:
 	$$\Omega_{\tiny\yng(2)}=\Omega^\circ_{\tiny\yng(2)}\sqcup\Omega^\circ_{\tiny\yng(3)}\sqcup\Omega^\circ_{\tiny\yng(4)}\sqcup\Omega^\circ_{\tiny\yng(5)}.$$
 	In fact, every Schubert variety is a disjoint union of Schubert cells.  See the problems at the end of this section for details.
 \end{example}

 We may generalize this construction to other bases than the standard basis $e_1,\ldots,e_n$, or more rigorously, using any \textit{complete flag}.  A \textbf{complete flag} is a chain of subspaces $$F_{\bullet}: 0=F_0\subset F_1\subset\cdots \subset F_n=\CC^n$$ where each $F_i$ has dimension $i$.  Then we define 
 $$\Omega_{\lambda}(F_\bullet)=\{V\in \Gr(n,k)\mid \dim(V\cap F_{n-k+i-\lambda_i})\ge i\}$$ and similarly for $\Omega^\circ_\lambda$.
 
 \begin{example}
 	The Schubert variety $\Omega_{\square}(F_\bullet)\subset \Gr(4,2)$ consists of the $2$-dimensional subspaces $V$ of $\CC^4$ for which $\dim(V\cap F_{2})\ge 1$.  Under the quotient map $\CC^4\to \PP^3$, this is equivalent to space of all lines in $\PP^3$ that intersect a given line in at least a point, which is precisely the variety we need for Question \ref{lines}.
 \end{example}
 
 \subsection{A note on flags}
 
     Why are chains of subspaces called \textit{flags}?  Roughly speaking, a flag on a flagpole consists of:
     \begin{itemize}
     	\item A point (the top of the pole),
     	\item A line passing through that point (the pole), 
     	\item A plane passing through that line (the plane containing the flag), and 
     	\item Space to put it in.
     \end{itemize}
     
     Mathematically, this is the data of a \textit{complete flag} in three dimensions.  However, higher-dimensional beings would require more complicated flags.  So in general, it is natural to define a complete flag in $n$-dimensional space $\mathbb{C}^n$ to be a chain of vector spaces $F_i$ of each dimension from $0$ to $n$, each containing the previous, with $\dim(F_i)=i$ for all $i$.  A \textbf{partial flag} is a chain of subspaces in which only some of the possible dimensions are included.

  \subsection{Problems}
  
  \begin{enumerate}
  	\item \textbf{Projective space is a Grassmannian:} Show that every projective space $\PP^m$ is a Grassmannian.  What are $n$ and $k$?
  	
  	\item \textbf{Schubert cells in $\PP^m$:} What are the Schubert cells in $\PP^m$?  Express your answer in homogeneous coordinates.
  	
  	\item \textbf{Schubert varieties in $\PP^m$:}   What are the Schubert varieties in $\PP^m$, thought of as a Grassmannian?  Why are they the closures of the Schubert cells in the topology on $\PP^m$?
    
    \item \textbf{Schubert varieties vs.\ Schubert cells:}	Show that every Schubert variety is a disjoint union of Schubert cells.  Describe which Schubert cells are contained in $\Omega_\lambda$ in terms of partitions.
    
    \item \textbf{Extreme cases:}  Describe $\Omega_{\emptyset}$ and $\Omega_{B}$ where $B$ is the entire ambient rectangle.  What are their dimensions?
    
    \item\label{previous} \textbf{Intersecting Schubert Varieties:} Show that, by choosing four different flags $F^{(1)}_{\bullet}$, $F^{(2)}_{\bullet}$, $F^{(3)}_{\bullet}$, $F^{(4)}_{\bullet}$, Question \ref{lines} becomes equivalent to finding the intersection of the Schubert varieties $$\Omega_{\square}(F^{(1)}_\bullet)\cap \Omega_{\square}(F^{(2)}_\bullet)\cap \Omega_{\square}(F^{(3)}_\bullet)\cap \Omega_{\square}(F^{(4)}_\bullet).$$
    
    \item \textbf{A Variety of Varieties:}  Translate the simple intersection problems of lines passing through two points, points contained in two lines, and so on into problems about intersections of Schubert varieties, as we did for Question \ref{lines} in Problem \ref{previous} above.  What does Question \ref{complicated} become?
    
    \item\label{2-planes} \textbf{More complicated flag conditions:} In $\PP^4$, let $2$-planes $A$ and $B$ intersect in a point $X$, and let $P$ and $Q$ be distinct points different from $X$.  Let $S$ be the set of all $2$-planes $C$ that contain both $P$ and $Q$ and intersect $A$ and $B$ each in a line.  Express $S$ as an intersection of Schubert varieties in $\Gr(5,3)$, in each of the following cases:
    \begin{enumerate}
    	\item When $P$ is contained in $A$ and $Q$ is contained in $B$;
    	\item When neither $P$ nor $Q$ lie on $A$ or $B$.
    \end{enumerate}
    
  \end{enumerate}
  
  \pagebreak
  
%%%%%%%%%%% DAY 3 %%%%%%%%%%%%%%%%%
  
\section{Variation 1: Intersections of Schubert varieties in the Grassmannian}\label{sec:variation1}

  In the previous section, we saw how to express certain linear intersection problems as intersections of Schubert varieties in a Grassmannian.  We now will build up the machinery needed to obtain a combinatorial rule for computing these intersections, known as the \textbf{Littlewood-Richardson rule}.  
 
  Both the geometric and combinatorial aspects of the Littlewood-Richardson rule are fairly complicated to prove, and we refer the reader to \cite{Fulton} for complete proofs.  In this exposition we will focus more on the applications and intuition behind the rule.  
  
  The Littlewood-Richardson rule is particularly nice in the case of zero-dimensional intersections.  In particular, given a list of generic flags $F^{(i)}_\bullet$ in $\CC^n$ for $i=1,\ldots,r$, let $\lambda^{(1)},\ldots,\lambda^{(r)}$ be partitions with $$\sum |\lambda^{i}|=k(n-k).$$  Then the intersection $$\bigcap \Omega_{\lambda^{i}}(F^{(i)}_\bullet)$$ is zero-dimensional, consisting of exactly $c^{B}_{\lambda^{(1)},\ldots,\lambda^{(r)}}$ points of $\Gr(n,k)$, where $B$ is the ambient rectangle and $c^{B}_{\lambda^{(1)},\ldots,\lambda^{(r)}}$ is a certain \textbf{Littlewood-Richardson coefficient}, defined in Section \ref{sec:LRrule}.
  
  When we refer to a ``generic'' choice of flags, we mean that we are choosing from an open dense subset of the \textit{flag variety}.  This will be made more precise in Section \ref{sec:variation2}.
  
  In general, the Littlewood Richardson rule computes products of Schubert classes in the cohomology ring of the Grassmannian, described in Section \ref{sec:cohomology} below, which corresponds with (not necessarily zero-dimensional) intersections of Schubert varieties.   To gain intuition for these intersections, we follow \cite{Fulton} and first simplify even further, to the case of two flags that intersect \textit{transversely}.
 
  \subsection{Opposite and transverse flags, genericity}
  
  Two subspaces of $\CC^n$ are said to be \textit{transverse} if their intersection has the ``expected dimension''.  For instance, two $2$-dimensional subspaces of $\CC^3$ are expected to have a $1$-dimensional intersection; only rarely is their intersection $2$-dimensional (when the two planes coincide).  More rigorously:
  
  \begin{definition}
  	Two subspaces $V$ and $W$ of $\CC^n$ are \textbf{transverse} if $$\dim(V\cap W)=\max(0, \dim(V) +\dim(W) -n).$$  Equivalently, if $\mathrm{codim}(V)$ is defined to be $n-\dim(V)$, then $$\mathrm{codim}(V\cap W)=\min(n,\mathrm{codim}(V)+\mathrm{codim}(W)).$$
  \end{definition}
  
  \begin{exercise}
  	Verify that the two definitions above are equivalent.
  \end{exercise}

  We say two flags $F^{(1)}_\bullet$ and $F^{(2)}_\bullet$ are \textbf{transverse} if every pair of subspaces $F^{(1)}_i$ and $F^{(2)}_j$ are transverse.  In fact, a weaker condition suffices:
  
  \begin{lemma}\label{lem:opposite}
  	Two complete flags $F_\bullet,E_\bullet\subset \CC^n$ are transverse if and only if $F_{n-i}\cap E_i=\{0\}$ for all $i$.
  \end{lemma}
  
  \begin{proof}[Proof Sketch.]
  	The forward direction is clear.  For the reverse implication, we can take the quotient of both flags by the one-dimensional subspace $E_1$ and induct on $n$.
  \end{proof}

  Define the \textbf{standard flag} $F_\bullet$ to be the flag in which $F_i=\langle e_1,\ldots, e_i\rangle $, and similarly define the \textbf{opposite flag} $E_\bullet$ by $E_i=\langle e_n,\ldots, e_{n-i+1} \rangle$.  It is easy to check that these flags $F_\bullet$ and $E_\bullet$ are transverse.  Furthermore, we shall see that every pair of transverse flags can be mapped to this pair, as follows.  Consider the action of $\GL_n(\CC)$ on $\CC^n$ by standard matrix multiplication, and note that this gives rise to an action on flags and subspaces, and subsequently on Schubert varieties as well.  
  
  \begin{lemma}\label{lem:transverse}
  	For any pair of transverse flags $F'_\bullet$ and $E'_\bullet$, there is an element $g\in \GL_n$ such that $gF'_\bullet=F_\bullet$ and $gE'_\bullet=E_\bullet$, where $F_\bullet$ and $E_\bullet$ are the standard and opposite flags. 
  \end{lemma}
  
  The proof of this lemma is left as an exercise to the reader (see the Problems section below).  The important corollary is that to understand the intersection of the Schubert varieties $\Omega_{\lambda}(F'_\bullet)$ and $\Omega_{\mu}(E'_\bullet)$, it suffices to compute the intersection $\Omega_\lambda(F_\bullet)\cap \Omega_{\lambda}(E_\bullet)$ and multiply the results by the appropriate matrix $g$.
  
  So, when we consider the intersection of two Schubert varieties with respect to transverse flags, it suffices to consider the standard and opposite flags $F_\bullet$ and $E_\bullet$.  We use this principle in the \textbf{duality theorem} below, which tells us exactly when the intersection of $\Omega_\lambda(F_\bullet)$ and $\Omega_{\mu}(E_\bullet)$ is nonempty. 
  
  \subsection{Duality theorem}
  
  \begin{definition}
  	Two partitions $\lambda=(\lambda_1,\ldots,\lambda_k)$ and $\mu=(\mu_1,\ldots,\mu_k)$ are \textbf{complementary} in the $k\times (n-k)$ ambient rectangle if and only if $\lambda_i+\mu_{k+1-i}=n-k$ for all $i$.  In this case we write $\mu^c=\lambda$.
  \end{definition}
  
  In other words, if we rotate the Young diagram of $\mu$ and place it in the lower right corner of the ambient rectangle, its complement is $\lambda$.   Below, we see that $\mu=(3,2)$ is the complement of $\lambda=(4,2,1)$ in $\Gr(7,3)$.
  
  \begin{center}
  	\includegraphics{ImportantBox-eps-converted-to.pdf}
  \end{center}
  
  \begin{theorem}[Duality Theorem]
  	Let $F_\bullet$ and $E_\bullet$ be transverse flags in $\CC^n$, and let $\lambda$ and $\mu$ be partitions with $|\lambda|+|\mu|=k(n-k)$.  In $\Gr(n,k)$, the intersection $\Omega_\lambda(F_\bullet)\cap \Omega_{\mu}(E_\bullet)$ has $1$ element if $\mu$ and $\lambda$ are complementary partitions, and is empty otherwise.  Furthermore, if $\mu$ and $\lambda$ are any partitions with $\mu_{k+1-i}+\lambda_i>n-k$ for some $i$ then $\Omega_\lambda(F_\bullet)\cap \Omega_{\mu}(E_\bullet)=\emptyset$.
  \end{theorem}
  
  We can use a reversed form of row reduction to express the Schubert varieties with respect to the opposite flag, and then the Schubert cells for the complementary partitions will have their $1$'s in the same positions, as in the example below.  Their unique intersection will be precisely this matrix of $1$'s with $0$'s elsewhere.
  
   $$\left[\begin{array}{ccccccc}
   0 & 0 & 0 & 0 & 0 & 0 & 1 \\
   0 & 0 & 0 & 1 & \ast & \ast & 0 \\
   0 & 1 & \ast & 0 & \ast & \ast & 0 \\
   \end{array}\right]  \hspace{2cm}  
   \left[\begin{array}{ccccccc}
   \ast & 0 & \ast & 0 & \ast & \ast & 1 \\
   \ast & 0 & \ast & 1 & 0    & 0    & 0 \\
   \ast & 1 & 0    & 0 & 0    & 0    & 0 \\
   \end{array}\right]$$
  
  We now give a more rigorous proof below, which follows that in \cite{Fulton} but with a few notational differences.
  
  \begin{proof}
    We prove the second claim first: if for some $i$ we have $\mu_{k+1-i}+\lambda_i>n-k$ then $\Omega_{\lambda}(F_\bullet)\cap \Omega_\mu(E_\bullet)$ is empty.  Assume for contradiction that there is a subspace $V$ in the intersection.  We know $\dim(V)=k$, and also 
    \begin{equation}\label{eqn1}
    \dim(V\cap \langle e_1,e_2,\ldots, e_{n-k+i-\lambda_i} \rangle)\ge i,
    \end{equation}
    $$\dim(V\cap \langle e_{n},e_{n-1},\ldots, e_{n+1-(n-k+(k+1-i)-\mu_{k+1-i})} \rangle)\ge k+1-i.$$
    Simplifying the last subscript above, and reversing the order of the generators, we get
    \begin{equation}\label{eqn2}
    \dim(V\cap \langle e_{i+\mu_{k+1-i}},\ldots, e_{n-1},e_n\rangle)\ge k+1-i.
    \end{equation}
    
    Notice that $i+\mu_{k+1-i}>n-k+i-\lambda_i$ by the condition $\mu_{k+1-i}+\lambda_i>n-k$, so the two subspaces we are intersecting with $V$ in equations (\ref{eqn1}) and (\ref{eqn2}) are disjoint.  It follows that $V$ has dimension at least $k+1-i+i=k+1$, a contradiction.  Thus $\Omega_{\lambda}(F_\bullet)\cap \Omega_\mu(E_\bullet)$ is empty in this case.
    
    Thus, if $|\lambda|+|\mu|=k(n-k)$ and $\lambda$ and $\mu$ are not complementary, then the intersection is empty as well, since the inequality $\mu_{k+1-i}+\lambda_i>n-k$ must hold for some $i$.
    
    Finally, suppose $\lambda$ and $\mu$ are complementary.  Then equations (\ref{eqn1}) and (\ref{eqn2}) still hold, but now $n-k+i-\lambda_i=i+\mu_{n+1-i}$ for all $i$.  Thus $\dim(V\cap\langle e_{i+\mu_{n+1-i}}\rangle)=1$ for all $i=1,\ldots,k$, and since $V$ is $k$-dimensional it must equal the span of these basis elements, namely $$V=\langle e_{1+\mu_n},e_{2+\mu_{n-1}},\ldots e_{k+\mu_{n+1-k}} \rangle.$$ This is the unique solution.
  \end{proof}

  \begin{example}
  	We now can give a rather high-powered proof that there is a unique line passing through any two distinct points in $\PP^n$.  As before, we work in one higher dimensional affine space and consider $2$-planes in $\CC^{n+1}$. Working in $\Gr(n+1,2)$, the two distinct points become two distinct one-dimensional subspaces $F_1$ and $E_1$ of $\CC^{n+1}$, and the Schubert condition that demands the $2$-dimensional subspace $V$ contains them is $$\dim(V\cap F_1)\ge 1,\hspace{1cm} \dim(V\cap E_1)\ge 1.$$ These are the Schubert conditions for a single-part partition $\lambda=(\lambda_1)$ where $(n+1)-2+1-\lambda_1=1$.  Thus $\lambda_1=n-1$, and we are intersecting the Schubert varieties $$\Omega_{(n-1)}(F_\bullet)\cap \Omega_{(n-1)}(E_\bullet)$$ where $F_\bullet$ and $E_\bullet$ are any two transverse flags extending $F_1$ and $E_1$.  Notice that $(n-1)$ and $(n-1)$ are complementary partitions in the $2\times (n-1)$ ambient rectangle (see Figure \ref{fig:two-row}), so by the Duality Theorem there is a unique point of $\Gr(n+1,2)$ in the intersection.  The conclusion follows.
    
    \begin{figure}[b]
    	\begin{center}
    		\includegraphics{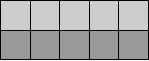}
    	\end{center}
    	
    	\caption{\label{fig:two-row} Two complimentary partitions of size $n-1$ filling the $n-1\times 2$ rectangle.}
    \end{figure}
    
  \end{example}

\subsection{Cell complex structure}

In order to prove the more general zero-dimensional Littlewood-Richardson rule and compute the Littlewood-Richardson coefficients, we need to develop more heavy machinery.  In particular, we need to understand the Grassmannian as a geometric object and compute its \textbf{cohomology}, an associated ring in which multiplication of certain generators will correspond to intersection of Schubert varieties.  (See \cite{Hatcher} for more details on all of the material in this section.)

The term \textit{Schubert cell} comes from the notion of a \textbf{cell complex} (also known as a CW complex) in algebraic topology.    An \textbf{$n$-cell} is a topological space homeomorphic to the open ball $|v|<1$ in $\RR^n$, and its associated \textbf{$n$-disk} is its closure $|v|\le 1$ in $\RR^n$.

 To construct a cell complex, one starts with a set of points called the \textbf{$0$-skeleton} $X^0$, and then attaches $1$-disks $D$ via continuous boundary maps from the boundary $\partial D$ (which consists of two points) to $X^0$.  The result is a \textbf{$1$-skeleton} $X^1$. 
 
 This can then be extended to a $2$-skeleton by attaching $2$-disks $D$ via maps from the boundary $\partial D$ (which is a circle) to $X^1$.  In general the $n$-skeleton $X^n$ is formed by attaching a set of $n$-disks to $X^{n-1}$ along their boundaries.

  More precisely, to form $X^n$ from $X^{n-1}$, we start with a collection of $n$-disks $D^n_\alpha$ and continuous attaching maps $\varphi_\alpha:\partial D_\alpha^n\to X^{n-1}$.  Then $$X^n=\frac{X^{n-1}\sqcup \bigsqcup_\alpha D_\alpha^n}{\sim}$$ where $\sim$ is the identification $x\sim \varphi_\alpha(x)$ for $x\in \partial D^n_\alpha$.  The cell complex is $X=\bigcup_n X^n$, which may be simply $X=X^n$ if the process stops at stage $n$.  By the construction, the points of $X^0$ along with the open $i$-cells associated with the $i$-disks in $X^i$ for each $i$ are disjoint and cover the cell complex $X$.  The topology is given by the rule that $A\subset X$ is open if and only if $A\cap X^n$ is open in $X^n$ for all $n$, where the topology on $X^n$ is given by the usual Euclidean topology on $\mathbb{R}^n$.

\begin{example}\label{realproj}
   The real projective plane $\PP^2_{\mathbb{R}}$ has a cell complex structure in which $X^0=\{(0:0:1)\}$ is a single point, $X^1=X^0\sqcup \{(0:1:\ast)\}$ is topologically a circle formed by attaching a $1$-cell to the point at both ends, and then $X^2$ is formed by attaching a $2$-cell $\mathbb{R}^2$ to the circle such that the boundary wraps around the $1$-cell twice.  This is because the points of the form $(1:xt:yt)$ as $t\to \infty$ and as $t\to -\infty$ both approach the same point in $X^1$, so the boundary map must be a $2$-to-$1$ mapping.
\end{example}

\begin{example}
	The \textit{complex} projective plane $\PP^2_{\CC}$ has a simpler cell complex structure, consisting of starting with a single point $X^0=\{(0:0:1)\}$, and then attaching a $2$-cell (a copy of $\CC=\RR^2$) like a balloon to form $X^2$.  A copy of $\CC^2=\RR^4$ is then attached to form $X^4$.
\end{example}

The Schubert cells give a cell complex structure on the Grassmannian.  For a complete proof of this, see \cite{Tuomas}, section 3.2.  We sketch the construction below.

 Define the $0$-skeleton $X^0$ to be the $0$-dimensional Schubert variety $\Omega_{((n-k)^k)}$.  Define $X^2$ to be $X^0$ along with the $2$-cell (since we are working over $\CC$ and not $\RR$) given by $\Omega_{((n-k)^{k-1},n-k-1)}^\circ$, and the attaching map given by the closure in $\Gr(n,k)$.  Note that the partition in this step is formed by removing a single corner square from the ambient rectangle.  

Then, $X^4$ is formed by attaching the two four-cells given by removing two outer corner squares in both possible ways, giving either $\Omega_{((n-k)^{k-2},n-k-1,n-k-1)}^\circ$ or $\Omega_{((n-k)^{k-1},n-k-2)}^\circ$.  We can continue in this manner with each partition size to define the entire cell structure, $X^0\subset X^2\subset\cdots \subset X^{2k(n-k)}$. 

\begin{example}
	We have $$\Gr(4,2)=\Omega^\circ_{\tiny\yng(2,2)}\sqcup\Omega^\circ_{\tiny\yng(2,1)}\sqcup\Omega^\circ_{\tiny\yng(2)}\sqcup\Omega^\circ_{\tiny\yng(1,1)}\sqcup\Omega^\circ_{\tiny\yng(1)}\sqcup\Omega^\circ_{\emptyset},$$
	forming a cell complex structure in which $X^0=\Omega^\circ_{\tiny\yng(2,2)}$, $X^2$ is formed by attaching $\Omega^\circ_{\tiny\yng(2,1)}$, $X^4$ is formed by attaching $\Omega^\circ_{\tiny\yng(2)}\sqcup\Omega^\circ_{\tiny\yng(1,1)}$, $X^6$ is formed by attaching $\Omega^\circ_{\tiny\yng(1)}$, and  $X^8$ is formed by attaching $\Omega^\circ_{\emptyset}$.
\end{example}

\subsection{Cellular homology and cohomology}\label{sec:cohomology}

   For a CW complex $X=X^0\subset \cdots \subset X^n$, define $$C_k=\mathbb{Z}^{\#k\text{-cells}},$$
   the free abelian group generated by the $k$-cells $B_\alpha^{(k)}=(D_\alpha^{(k)})^\circ$.
   
   Define the \textbf{cellular boundary map} $d_{k+1}:C_{k+1}\to C_k$ by $$d_{k+1}(B_{\alpha}^{(k+1)})=\sum_{\beta} \mathrm{deg}_{\alpha\beta}\cdot B_{\beta}^{(k)},$$ where $\mathrm{deg}_{\alpha\beta}$ is the \textit{degree} of the composite map $$\partial \overline{B_\alpha^{(k+1)}}\to X^k\to \overline{B_\beta^{(k)}}.$$  The first map above is the cellular attaching map from the boundary of the closure of the ball $B_\alpha^{(k+1)}$ to the $k$-skeleton, and the second map is the quotient map formed by collapsing $X^k\setminus B_\beta^{(k)}$ to a point.  The composite is a map from a $k$-sphere to another $k$-sphere, which has a \textbf{degree}, whose precise definition we omit here and refer the reader to \cite{Hatcher}, section 2.2, p.\ 134.  As one example, the $2$-to-$1$ attaching map described in Example \ref{realproj} for $\PP^2_\mathbb{R}$ has degree $2$.
   
   It is known that the cellular boundary maps make the groups $C_k$ into a \textbf{chain complex}: a sequence of maps
   $$0\to C_n\xrightarrow{d_n}C_{n-1}\xrightarrow{d_{n-1}}C_{n-2}\to \cdots \to C_1 \xrightarrow{d_1}C_0\to 0$$ for which $d_i\circ d_{i+1}=0$ for all $i$.  This latter condition implies that the image of the map $d_{i+1}$ is contained in the kernel of $d_i$ for all $i$, and so we can consider the quotient groups $$H_i(X)=\ker(d_i)/\mathrm{im}(d_{i+1})$$ for all $i$.  These quotients are abelian groups called the \textbf{cellular homology groups} of the space $X$.
   
   \begin{example}
   	Recall that $\PP^2_{\CC}$ consists of a point, a $2$-cell, and a $4$-cell.  So, its cellular chain complex is:
   	$$\cdots \to 0\to 0 \to 0 \to \ZZ\to 0 \to \ZZ \to 0 \to \ZZ\to 0$$
   	and the homology groups are $H_0=H_2=H_4=\ZZ$, $H_1=H_3=0$.
   	
   	On the other hand, in $\PP^2_{\RR}$, the chain complex looks like:
   	$$0\to \ZZ\to \ZZ \to \ZZ\to 0$$ where the first map $\ZZ\to \ZZ$ is multiplication by $2$ and the second is the $0$ map, due to the degrees of the attaching maps.  It follows that $H_2=0$, $H_1=\mathbb{Z}/2\ZZ$, and $H_0=\ZZ$.
   \end{example}

   We can now define the \textbf{cellular cohomology} by dualizing the chain complex above.  In particular, define $$C^k=\mathrm{Hom}(C_k,\mathbb{Z})=\{\text{group homomorphisms } f:C_k\to \mathbb{Z}\}$$ for each $k$, and define the \textbf{coboundary maps} $d_k^\ast:C^{k-1}\to C^k$ by $$d_k^\ast f(c)=f(d_k(c))$$ for any $f\in C^k$ and $c\in C_k$.  Then the coboundary maps form a \textbf{cochain complex}, and we can define the cohomology groups to be the abelian groups $$H^i(X)=\ker(d_{i+1}^\ast)/\mathrm{im}(d_i^\ast)$$ for all $i$.
   
   \begin{example}
   	The cellular cochain complex for $\PP^2_{\CC}$ is $$0\to \ZZ\to 0 \to \ZZ \to 0 \to \ZZ \to 0 \to 0 \to 0 \to \cdots$$
   	and so the cohomology groups are $H^0=H^2=H^4=\ZZ$, $H^1=H^3=0$.
   \end{example}
   
   Finally, the direct sum of the cohomology groups $$H^\ast(X)=\bigoplus H^i(X)$$ has a ring structure given by the \textbf{cup product} (\cite{Hatcher}, p.\ 249), which is the dual of the ``cap product'' (\cite{Hatcher}, p.\ 239) on homology and roughly corresponds to taking intersection of cohomology classes in this setting.
   
   In particular, there is an equivalent definition of cohomology on the Grassmannian known as the \textit{Chow ring}, in which cohomology classes in $H^\ast(X)$ are equivalence classes of algebraic subvarieties under \textbf{birational equivalence}. (See \cite{FultonIntersectionTheory}, Sections 1.1 and 19.1.) In other words, deformations under rational families are still equivalent:  in $\PP^2$, for instance, the family of algebraic subvarieties of the form $xy-tz^2=0$ as $t\in \CC$ varies are all in one equivalence class, even as $t\to 0$ and the hyperbola degenerates into two lines. 
   
   The main fact we will be using under this interpretation is the following, which we state without proof.  (See \cite{Fulton}, Section 9.4 for more details.)
   
   \begin{theorem}
   	The cohomology ring $H^\ast(\Gr(n,k))$ has a $\ZZ$-basis given by the classes $$\sigma_\lambda:=[\Omega_\lambda(F_\bullet)]\in H^{2|\lambda|}(\Gr(n,k))$$ for $\lambda$ a partition fitting inside the ambient rectangle.  The cohomology $H^\ast(\Gr(n,k))$ is a graded ring, so $\sigma_\lambda\cdot \sigma_\mu\in H^{2|\lambda|+2|\mu|}(\Gr(n,k))$, and we have $$\sigma_\lambda\cdot \sigma_\mu=[\Omega_\lambda(F_\bullet)\cap \Omega_\mu(E_\bullet)]$$ where $F_\bullet$ and $E_\bullet$ are the standard and opposite flags.
   \end{theorem} 
   
   Note that $\sigma_\lambda$ is independent of the choice of flag $F_\bullet$, since any two Schubert varieties of the same partition shape are rationally equivalent via a change of basis.  
   
   We can now restate the intersection problems in terms of multiplying Schubert classes.  In particular, if $\lambda^{(1)},\ldots,\lambda^{(r)}$ are partitions with $\sum_i |\lambda^{(i)}|=k(n-k)$, then $$\sigma_{\lambda^{(1)}}\cdots \sigma_{\lambda^{(r)}}\in H^{k(n-k)}(\Gr(n,k))$$ and there is only one generator of the top cohomology group, namely $\sigma_B$ where $B$ is the ambient rectangle.  This is the cohomology class of the single point $\Omega_B(F_\bullet)$ for some flag $F_\bullet$.  Thus the intersection of the Schubert varieties $\Omega_{\lambda^{(1)}}(F^{(1)}_\bullet)\cap \cdots \cap \Omega_{\lambda^{(r)}}(F^{(r)}_\bullet)$ is rationally equivalent to a finite union of points, the number of which is the coefficient $c^{B}_{\lambda^{(1)},\ldots,\lambda^{(r)}}$ in the expansion $$\sigma_{\lambda^{(1)}}\cdots \sigma_{\lambda^{(r)}}=c^{B}_{\lambda^{(1)},\ldots,\lambda^{(r)}}\sigma_B.$$  For a sufficiently general choice of flags, the $c^{B}_{\lambda^{(1)},\ldots,\lambda^{(r)}}$ points in the intersection are distinct with no multiplicity.
   
   In general, we wish to understand the coefficients that we get upon multiplying Schubert classes and expressing the product back in the basis $\{\sigma_\lambda\}$ of Schubert classes.
   
   \begin{example}\label{ExampleLines}
   	In Problem \ref{previous}, we saw that Question \ref{lines} can be rephrased as computing the size of the intersection $$\Omega_{\square}(F^{(1)}_\bullet)\cap \Omega_{\square}(F^{(2)}_\bullet)\cap \Omega_{\square}(F^{(3)}_\bullet)\cap \Omega_{\square}(F^{(4)}_\bullet)$$ for a given generic choice of flags $F^{(1)}_\bullet,\ldots, F^{(4)}_\bullet$.  By the above analysis, we can further reduce this problem to computing the coefficient $c$ for which  $$\sigma_{\square}\cdot \sigma_\square\cdot \sigma_\square\cdot \sigma_\square=c \cdot \sigma_{\scalebox{0.3}{\yng(2,2)}}$$ in $H^\ast(\Gr(4,2))$. 
   \end{example}
   
\subsection{Connection with symmetric functions}
 
 We can model the cohomology ring $H^\ast(\Gr(n,k))$ algebraically as a quotient of the ring of \textbf{symmetric functions}.  We only cover the essentials of symmetric function theory for our purposes here, and refer the reader to Chapter 7 of \cite{Stanley}, or the books \cite{Macdonald} or \cite{Sagan} for more details, or to \cite{Fulton} for the connection between $H^\ast(\Gr(n,k))$ and the ring of symmetric functions.
 
 \begin{definition}
 	The ring of \textit{symmetric functions} $\Lambda_{\mathbb{C}}(x_1,x_2,\ldots)$ is the ring of bounded-degree formal power series $f\in \mathbb{C}[[x_1,x_2,\ldots]]$ which are symmetric under permuting the variables, that is, $$f(x_1,x_2,\ldots)=f(x_{\pi(1)},x_{\pi(1)},\ldots)$$ for any permutation $\pi:\mathbb{Z}_+\to\mathbb{Z}_+$ and $\mathrm{deg}(f)<\infty$. 
 \end{definition}
 
 For instance, $x_1^2+x_2^2+x_3^2+\cdots$ is a symmetric function of degree $2$.
 
 The most important symmetric functions for Schubert calculus are the \textit{Schur functions}.  They can be defined in many equivalent ways, from being characters of irreducible representations of $S_n$ to an expression as a ratio of determinants.  We use the combinatorial definition here, and start by introducing some common terminology involving Young tableaux and partitions.
 
 \begin{definition}
 	A \textbf{skew shape} is the difference $\nu/\lambda$ formed by cutting out the Young diagram of a partition $\lambda$ from the strictly larger partition $\nu$.  A skew shape is a \textbf{horizontal strip} if no column contains more than one box.  
 \end{definition}
 
 \begin{definition}
 	A \textbf{semistandard Young tableau (SSYT)} of shape $\nu/\lambda$ is a way of filling the boxes of the Young diagram of $\nu/\lambda$ with positive integers so that the numbers are weakly increasing across rows and strictly increasing down columns.  An SSYT has \textbf{content} $\mu$ if there are $\mu_i$ boxes labeled $i$ for each $i$.  The \textbf{reading word} of the tableau is the word formed by concatenating the rows from bottom to top.
 \end{definition}  
 
 The following is a semistandard Young tableau of shape $\nu/\lambda$ and content $\mu$ where $\nu=(6,5,3)$, $\lambda=(3,2)$, and $\mu=(4,2,2,1)$.  Its reading word is $134223111$.
 
 \begin{center}
 	\includegraphics{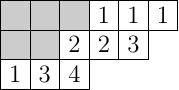}
 \end{center}
 
 \begin{definition}
 	Let $\lambda$ be a partition.  Given a semistandard Young tableau $T$ of shape $\lambda$, define $x^T=x_1^{m_1}x_2^{m_2}\cdots$ where $m_i$ is the number of $i$'s in $T$.  The \textbf{Schur function} for a partition $\lambda$ is the symmetric function defined by $$s_\lambda=\sum_{T} x^T$$ where the sum ranges over all SSYT's $T$ of shape $\lambda$.
 \end{definition}
 
 \begin{example}
 	For $\lambda=(2,1)$, the tableaux 
 	$$\young(11,2)\hspace{0.5cm} \young(12,2) \hspace{0.5cm} \young(11,3)\hspace{0.5cm} \young(12,3)\hspace{0.5cm} \young(13,2)\hspace{0.2cm}\cdots $$
 	are a few of the infinitely many SSYT's of shape $\lambda$.  Thus we have $$s_{\lambda}=x_1^2x_2+x_1x_2^2+x_1^2x_3+2x_1x_2x_3+\cdots.$$
 \end{example}
 
 It is well-known that the Schur functions $s_\lambda$ are symmetric and form a vector space basis of $\Lambda(x_1,x_2,\ldots)$ as $\lambda$ ranges over all partitions.  The key fact that we will need is that they allow us to understand the cohomology ring $H^\ast(\Gr(n,k))$, as follows.
 
 \begin{theorem}\label{thm:isom}
 	There is a ring isomorphism $$H^\ast(\Gr(n,k))\isom \Lambda(x_1,x_2,\ldots)/(s_\lambda| \lambda\not\subset B)$$ where $B$ is the ambient rectangle and $(s_\lambda|\lambda\not\subset B)$ is the ideal generated by the Schur functions.  The isomorphism sends the Schubert class $\sigma_\lambda$ to the Schur function $s_\lambda$.
 \end{theorem}
 
 This is a pivotal theorem in the study of the Grassmannian, since it allows us to compute in the cohomology ring simply by working with symmetric polynomials. In particular, multiplying Schur functions corresponds to multiplying cohomology classes, which in turn gives us information about intersections of Schubert varieties.
 
  As an approach to prove Theorem \ref{thm:isom}, note that sending $\sigma_\lambda$ to $s_\lambda$ is an isomorphism of the underlying vector spaces, since on the right hand side we have quotiented by the Schur functions whose partition does not fit inside the ambient rectangle.  So, it remains to show that this isomorphism respects the multiplications in these rings, taking cup product to polynomial multiplication. 
 
 An important first step is the \textbf{Pieri Rule}.  For Schur functions, this tells us how to multiply a one-row shape by any other partition: $$s_{(r)}\cdot s_{\lambda}=\sum_{\nu/\lambda \text{ horz. strip of size }r} s_\nu.$$  We wish to show that the same relation holds for the $\sigma_\lambda$'s, that is, that $$\sigma_{(r)}\cdot \sigma_{\lambda}=\sum_{\nu/\lambda \text{ horz. strip of size }r} \sigma_\nu,$$ where the sum on the right is restricted to partitions $\nu$ fitting inside the ambient rectangle.  Note that we do not need this restriction for general Schur functions, but in the cohomology ring we are considering the quotient by partitions not fitting inside the ambient rectangle, so the two expansions above are not exactly the same.
 
 Note that, by the Duality Theorem, we can multiply both sides of the above relation by $\sigma_{\mu^c}$ to extract the coefficient of $\sigma_{\mu}$ on the right hand side.  So, the Pieri Rule is equivalent to the following restatement:
 
 \begin{theorem}[Pieri Rule]
 	Let $\lambda$ and $\mu$ be partitions with $|\lambda|+|\mu|=k(n-k)-r$. Then if $F_\bullet$, $E_\bullet$, and $H_\bullet$ are three sufficiently general flags then the intersection $$\Omega_\lambda(F_\bullet)\cap \Omega_{\mu}(E_\bullet) \cap \Omega_{(r)}(H_\bullet)$$ has $1$ element if $\mu^c/\lambda$ is a horizontal strip, and it is empty otherwise.
 \end{theorem}
 
 \begin{proof}[Sketch of Proof.]  We can set $F_\bullet$ and $E_\bullet$ to be the standard and opposite flags and $H_\bullet$ a generic flag distinct from $F_\bullet$ or $E_\bullet$.  We can then perform a direct analysis similar to that in the Duality Theorem.  See \cite{Fulton} for full details.
 \end{proof}

 Algebraically, the Pieri rule suffices to show the ring isomorphism, because the Schur functions $s_{(r)}$ and corresponding Schubert classes $\sigma_{(r)}$ form an algebraic set of generators for their respective rings.  Therefore, to intersect Schubert classes we simply have to understand how to multiply Schur functions.

   \subsection{The Littlewood-Richardson rule}\label{sec:LRrule}
   
   The combinatorial rule for multiplying Schur functions, or Schubert classes, is called the \textbf{Littlewood-Richardson Rule}.  To state it, we need to introduce a few new notions.
   
   \begin{definition}
   	A word $w_1w_2\cdots w_n$ (where each $w_i\in \{1,2,3,\ldots\}$) is \textbf{Yamanouchi} (or \textbf{lattice} or \textbf{ballot}) if every suffix $w_k w_{k+1}\cdots w_n$ contains at least as many letters equal to $i$ as $i+1$ for all $i$.  
   \end{definition}
     
   For instance, the word $231211$ is Yamanouchi, because the suffixes $1$, $11$, $211$, $1211$, $31211$, and $231211$ each contain at least as many $1$'s as $2$'s, and at least as many $2$'s as $3$'s.
     
   \begin{definition}
   	 A \textbf{Littlewood-Richardson tableau} is a semistandard Young tableau whose reading word is Yamanouchi.
   \end{definition}

    \begin{figure}[h]
    	
     \begin{center}
     	\includegraphics{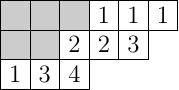}
     \end{center}
     \caption{\label{fig:LR} An example of a skew Littlewood-Richardson tableau.}
    \end{figure}
    
   \begin{exercise}
   	The example tableau in Figure \ref{fig:LR} is \textbf{not} Littlewood-Richardson.  Why?  Can you find a tableau of that shape that is?
   \end{exercise}
   
   \begin{definition}
   	A sequence of skew tableaux $T_1, T_2,\ldots$ form a \textbf{chain} if their shapes do not overlap and $$T_1\cup T_2\cup \cdots \cup T_i$$ is a partition shape for all $i$.
   \end{definition}
   
   We can now state the general Littlewood-Richardson rule.  We will refer the reader to \cite{Fulton} for a proof, as the combinatorics is quite involved.
   
   \begin{theorem}\label{thm:LRrule}
   	We have
   	$$s_{\lambda^{(1)}}\cdot\cdots \cdot  s_{\lambda^{(m)}}=\sum_{\nu} c^{\nu}_{\lambda^{(1)},\ldots,\lambda^{(m)}} s_\nu$$ 
   	where $c^{\nu}_{\lambda^{(1)},\ldots,\lambda^{(m)}}$ is the number of chains of Littlewood-Richardson tableaux of contents $\lambda^{(i)}$ with total shape $\nu$.
   \end{theorem}
   
   It is worth noting that in many texts, the following corollary is the primary focus, since the above theorem can be easily derived from the $m=2$ case stated below.
   
   \begin{corollary}
   	 We have $$s_\lambda s_\mu = \sum_{\nu} c^{\nu}_{\lambda\mu}s_\nu$$ where $c^{\nu}_{\lambda\mu}$ is the number of Littlewood-Richardson tableaux of skew shape $\nu/\lambda$ and content $\mu$.
   \end{corollary}
   
   \begin{proof}
   	By Theorem \ref{thm:LRrule}, $c^{\nu}_{\lambda\mu}$ is the number of chains of two Littlewood-Richardson tableaux of content $\lambda$ and $\mu$ with total shape $\nu$.  The first tableau of content $\lambda$ is a straight shape tableau, so by the Yamanouchi reading word condition and the semistandard condition, the top row can only contain $1$'s.  Continuing this reasoning inductively, it has only $i$'s in its $i$th row for each $i$.  Therefore the first tableau in the chain is the unique tableau of shape $\lambda$ and content $\lambda$.
   	
   	Thus the second tableau is a Littlewood-Richardson tableau of shape $\nu/\lambda$ and content $\mu$, and the result follows.
   \end{proof}

   As a consequence of Theorem \ref{thm:LRrule} and Theorem \ref{thm:isom}, in $H^\ast(\Gr(n,k))$ we have
   $$\sigma_{\lambda^{(1)}}\cdot\cdots \cdot  \sigma_{\lambda^{(m)}}=\sum_{\nu} c^{\nu}_{\lambda^{(1)},\ldots,\lambda^{(m)}} \sigma_\nu$$ where now the sum on the right is restricted to partitions $\nu$ fitting inside the ambient rectangle.  Note that by the combinatorics of the Littlewood-Richardson rule, the coefficients on the right are nonzero only if $|\nu|=\sum |\lambda^{(i)}|$, and so in the case of a zero-dimensional intersection of Schubert varieties, the only possible $\nu$ on the right hand side is the ambient rectangle $B$ itself.  Moreover, $\Omega_B(F_\bullet)$ is a single point of $\Gr(n,k)$ for any flag $F_\bullet$. The zero-dimensional Littlewood-Richardson rule follows as a corollary.
   
   \begin{theorem}[Zero-Dimensional Littlewood-Richardson Rule]  
   	 Let $B$ be the $k\times (n-k)$ ambient rectangle, and let $\lambda^{(1)},\ldots,\lambda^{(m)}$ be partitions fitting inside $B$ such that $|B|=\sum_i |\lambda^{(i)}|$.  Also let $F^{(1)}_\bullet,\ldots, F^{(m)}_\bullet$ be any $m$ generic flags.  Then $$c^{B}_{\lambda^{(1)},\ldots,\lambda^{(m)}}:=|\Omega_{\lambda^{(1)}}(F^{(1)}_\bullet)\cap \cdots \cap \Omega_{\lambda^{(m)}}(F^{(m)}_\bullet)|$$ is equal to the number of chains of Littlewood-Richardson tableaux of contents $\lambda^{(1)},\ldots, \lambda^{(m)}$ with total shape equal to $B$.
   \end{theorem}
   
   \begin{example}
   	 Suppose $k=3$ and  $n-k=4$.  Let $\lambda^{(1)}=(2,1)$, $\lambda^{(2)}=(2,1)$, $\lambda^{(3)}=(3,1)$, and $\lambda^{(4)}=2$.  Then there are five different chains of Littlewood-Richardson tableaux of contents $\lambda^{(1)},\ldots,\lambda^{(4)}$ that fill the $k\times (n-k)$ ambient rectangle, as shown in Figure \ref{fig:5}.  Thus $c^{B}_{\lambda^{(1)},\ldots,\lambda^{(4)}}=5$.
   \end{example}
   
   \begin{figure}
   	\begin{center}
   		\includegraphics{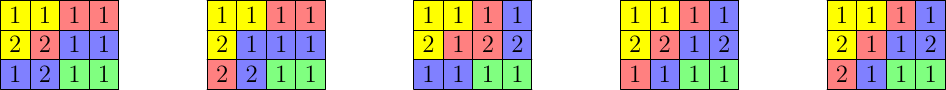}
   	\end{center}
   	\caption{\label{fig:5} The five chains of Littlewood-Richardson tableaux of contents $\lambda^{(1)}=(2,1)$, $\lambda^{(2)}=(2,1)$, $\lambda^{(3)}=(3,1)$, and $\lambda^{(4)}=2$ filling an ambient $3\times 4$ rectangle.}
   \end{figure}
   
   \begin{example}\label{LinesSolved}
   	 We can now solve Question \ref{lines}.  In Example \ref{ExampleLines} we showed that it suffices to compute the coefficient $c$ in the expansion $$\sigma_{\square}\cdot \sigma_\square\cdot \sigma_\square\cdot \sigma_\square=c \cdot \sigma_{\scalebox{0.3}{\yng(2,2)}}$$ in $H^\ast(\Gr(4,2))$.  This is the Littlewood-Richardson coefficient $c^{(2,2)}_{\square,\square,\square,\square}$.  This is the number of ways to fill a $2\times 2$ ambient rectangle with a chain of Littlewood-Richardson tableaux having one box each.  
   	 
   	 Since such a tableau can only have a single $1$ as its entry, we will label the entries with subscripts indicating the step in the chain to distinguish them.  We have two possibilities, as shown in Figure \ref{fig:2}.	 Therefore the coefficient is $2$, and so there are $2$ lines passing through four generic lines in $\PP^4$.
   \end{example}

\begin{figure}
	\begin{center}
		\includegraphics{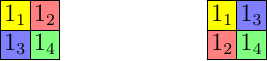}
	\end{center}
	\caption{\label{fig:2}The two tableaux chains used to enumerate the Littlewood-Richardson coefficient that answers Question \ref{lines}.}
\end{figure}

In Example \ref{LinesSolved}, we are in the special case in which each Littlewood-Richardson tableau in the chain has only one box, and so the only choice we have is the ordering of the boxes in a way that forms a chain.  We can therefore simply represent such a tableau by its indices instead, and the two tableaux of Figure \ref{fig:2} become $$\young(12,34)\hspace{0.5cm}\text{ and }\hspace{0.5cm}\young(13,24).$$

The two tableaux above are characterized by the property that the entries $1,2,3,4$ are used exactly once and the rows and columns are strictly increasing.  Such a tableau is called a \textit{standard Young tableaux}.

\begin{definition}
	A \textbf{standard Young tableau} of shape $\lambda$ with $|\lambda|=n$ is an SSYT of shape $\lambda$ in which the numbers $1,2,\ldots,n$ are each used exactly once.
\end{definition}

There is a well-known explicit formula, known as the \textbf{Hook length formula}, for the number of standard Young tableaux of a given shape, due to Frame, Robinson, and Thrall \cite{FrameRobinsonThrall}.  To state it we need the following definition.

\begin{definition}
	For a square $s$ in a Young diagram, define the \textbf{hook length} $$\hook(s)=\mathrm{arm}(s)+\mathrm{leg}(s)+1$$ where $\mathrm{arm}(s)$ is the number of squares strictly to the right of $s$ in its row and $\mathrm{leg}(s)$ is the number of squares strictly below $s$ in its column. 
\end{definition}

\begin{theorem} (Hook length formula.)
	The number of standard Young tableaux of shape $\lambda$ is $$\frac{|\lambda|!}{\prod_{s\in \lambda}\hook(s)}.$$
\end{theorem}

For example, if $\lambda=(2,2)$ then we have $\frac{4!}{3\cdot 2\cdot 2 \cdot 1}=2$ standard Young tableaux of shape $\lambda$, which matches our answer in Example \ref{LinesSolved}.

  \subsection{Problems}
  
  \begin{enumerate}
    
    \item \textbf{Prove Lemma \ref{lem:transverse}:} 	For any transverse flags $F'_\bullet$ and $E'_\bullet$, there is some $g\in \GL_n$ such that $gF'_\bullet=F_\bullet$ and $gE'_\bullet=E_\bullet$, where $F_\bullet$ and $E_\bullet$ are the standard and opposite flags. 
    
    \item \textbf{It's all Littlewood-Richardson:} Verify that the Duality Theorem and the Pieri Rule are both special cases of the Littlewood-Richardson rule.
    
    \item \textbf{An empty intersection:} Show that $$\Omega_{(1,1)}(F_\bullet)\cap \Omega_{(2)}(E_\bullet)=\emptyset$$ in $\Gr(4,2)$ for transverse flags $F_\bullet$ and $E_\bullet$.  What does this mean geometrically?
    
    \item \textbf{A nonempty intersection:}  Show that $$\Omega_{(1,1)}(F_\bullet)\cap \Omega_{(2)}(E_\bullet)$$ is nonempty in $\Gr(5,2)$.  (Hint: intersecting it with a certain third Schubert variety will be nonempty by the Littlewood-Richardson rule.)  What does this mean geometrically?
    
    \item \textbf{Problem \ref{2-planes} revisited:}  In $\PP^4$, suppose the $2$-planes $A$ and $B$ intersect in a point $X$, and $P$ and $Q$ are distinct points different from $X$.  Show that there is exactly one plane $C$ that contains both $P$ and $Q$ and intersect $A$ and $B$ each in a line as an intersection of Schubert varieties in $\Gr(5,3)$, in each of the following cases:
    \begin{enumerate}
    	\item When $P$ is contained in $A$ and $Q$ is contained in $B$;
    	\item When neither $P$ nor $Q$ lie on $A$ or $B$.
    \end{enumerate}
    
  	\item \textbf{That's a lot of $k$-planes:} Solve Question \ref{complicated} for a generic choice of flags as follows.
  	\begin{enumerate}
  		\item Verify that the problem boils down to computing the coefficient of $s_{((n-k)^{k})}$ in the product of Schur functions $s_{(1)}^{k(n-k)}$.
  		\item Use the Hook Length Formula to finish the computation.
  	\end{enumerate}
  	
  \end{enumerate}

\pagebreak
  
%%%%%%%%% DAY 4 %%%%%%%%%%%%%%%%%%

\section{Variation 2: The flag variety}\label{sec:variation2}
    
    For the content in this section, we refer to \cite{Manivel}, unless otherwise noted below.
    
    The (complete) \textbf{flag variety} (in dimension $n$) is the set of all complete flags in $\mathbb{C}^n$, with a Schubert cell decomposition similar to that of the Grassmannian.   In particular, given a flag $$V_{\bullet}: V_0\subset V_1\subset \cdots V_n=\CC^n,$$ we can choose $n$ vectors $v_1,\ldots,v_n$ such that the span of $v_1,\ldots,v_i$ is $V_i$ for each $i$, and list the vectors $v_i$ as row vectors of an $n\times n$ matrix.  We can then perform certain row reduction operations to form a different ordered basis $v_1^\prime,\ldots,v_n^\prime$ that also span the subspaces of the flag, but whose matrix entries consist of a permutation matrix of $1$'s, all $0$'s to the left and below each $1$, and arbitrary complex numbers in all other entries.
    
    For instance, say we start with the flag in three dimensions generated by the vectors $(0,2,3)$, $(1, 1, 4)$, and $(1, 2, -3)$.  The corresponding matrix is $$\left(\begin{array}{ccc} 0 & 2 & 3 \\ 1 & 1 & 4 \\ 1 & 2 & -3\end{array}\right).$$  We start by finding the leftmost nonzero element in the first row and scale that row so that this element is $1$.  Then subtract multiples of this row from the rows below it so that all the entries below that $1$ are $0$.  Continue the process on all further rows:
    
    $$\left(\begin{array}{ccc} 0 & 2 & 3 \\ 1 & 1 & 4 \\ 1 & 2 & -3\end{array}\right) \to \left(\begin{array}{ccc} 0 & 1 & 1.5 \\ 1 & 0 & 2.5 \\ 1 & 0 & -6\end{array}\right)\to \left(\begin{array}{ccc} 0 & 1 & 1.5 \\ 1 & 0 & 2.5 \\ 0 & 0 & 1\end{array}\right)$$
    
    It is easy to see that this process does not change the flag formed by the initial row spans, and that any two matrices in canonical form define different flags.  So, the flag variety is a cell complex consisting of $n!$ \textbf{Schubert cells} indexed by permutations.   For instance, one such open set in the $5$-dimensional flag variety is the open set given by all matrices of the form 
    $$\left(\begin{array}{ccccc} 
    0 & 1 & \ast &  \ast & \ast  \\ 
    1 & 0 & \ast &  \ast & \ast  \\ 
    0 & 0 & 0 & 0 & 1  \\
    0 & 0 & 1 & \ast & 0 \\
    0 & 0 & 0 & 1 & 0 \end{array}\right)$$
    We call this cell $X_{45132}^\circ$ because $4,5,1,3,2$ are the positions of the $1$'s from the right hand side of the matrix in order from top to bottom.  More rigorously, we define a Schubert cell as follows.
    
    \begin{definition}
    	Let $w\in S_n$ be a permutation of $\{1,\ldots,n\}$.  Then the \textbf{Schubert cell} of $w$ is defined by 
    	$$X_w^\circ=\{V_\bullet\in \Fl_n:\dim(V_p\cap F_q)=\#\{i\le p:w(i)\le q\} \text{ for all }p,q\}$$
    	where $F_\bullet$ is the standard flag generated by the unit vectors $e_{n+1-i}$.  In the matrix form above, the columns are ordered from right to left as before.
    \end{definition} 
    
    Note that, as in the case of the Grassmannian, we can choose a different flag $F_\bullet$ with respect to which we define our Schubert cell decomposition, and we define $X_w^\circ(F_\bullet)$ accordingly.
    
    The dimension of a Schubert cell $X_w$ is the number of $\ast$'s in its matrix, that is, the number of entries above and right of the pivot $1$ in its row and column.  The maximum number of $\ast$'s occurs when the permutation is $w_0=n(n-1)\cdots 3 2 1$, in which case the dimension of the open set $X_{w_0}$ is $n(n-1)/2$ (or $n(n-1)$ over $\RR$).  In general, it is not hard to see that the number of $\ast$'s in the set $X_w$ is the \textbf{inversion number} $\inv(w)$.  This is defined to be the number of pairs of entries $(w(i),w(j))$ of $w$ which are out of order, that is, $i<j$ and $w(i)>w(j)$.  Thus we have $$\dim(X_w^\circ)=\inv(w).$$
    
    \begin{example}
    	The permutation $w=45132$ has seven inversions.  (Can you find them all?)  We also see that $\dim(X_w^\circ)=7$, since there are seven $\ast$ entries in the matrix.
    \end{example}
    
    Another useful way to think of $\inv(w)$ is in terms of its \textbf{length}.  
    
    \begin{definition}
    	Define $s_1,\ldots,s_{n-1}\in S_n$ to be the \textit{adjacent transpositions} in the symmetric group, that is, $s_i$ is the permutation interchanging $i$ and $i+1$.  Then the \textbf{length} of $w$, written $\ell(w)$, is the smallest number $k$ for which there exists a decomposition $$w=s_{i_1}\cdots s_{i_k}.$$
    \end{definition}
    
    \begin{lemma}
    	We have $\ell(w)=\inv(w)$ for any $w\in S_n$.
    \end{lemma} 
    
    We will leave the proof of this lemma as an exercise to the reader in the Problems Section.  
    
    \subsection{Schubert varieties and the Bruhat order}
    
     By using the Pl\"ucker embeddings $\Gr(n,k)\hookrightarrow \PP^{\binom{n}{k}-1}$ for each $k$, we can embed $\Fl_n$ into the larger projective space $\PP^{2^n-1}$ whose entries correspond to the Pl\"ucker coordinates of each of the initial $k\times n$ submatrices of a given element of the flag variety.  This makes $\Fl_n$ a projective subvariety of $\PP^{2^n-1}$ (see \cite{Fulton} for more details), which in turn gives rise to a topology on $\Fl_n$, known as the Zariski topology.   Now, consider the closures of the sets $X_w^\circ$ in this topology. 
     
    \begin{definition}
    	The \textbf{Schubert variety} corresponding to a permutation $w\in S_n$ is $$X_w=\overline{X_w^\circ}.$$
    \end{definition}
    
    As in the Grassmannian, these Schubert varieties turn out to be disjoint unions of Schubert cells.  The partial ordering in which $X_w=\sqcup_{v\le w} X_v^\circ$ is called the \textbf{Bruhat order}, a well-known partial order on permutations.  We will briefly review it here, but we refer to \cite{BjornerBrenti} for an excellent introduction to Bruhat order.
    
    \begin{definition}
      The \textbf{Bruhat order} $\le$ on $S_n$ is defined by $v\le w$ if and only if, for every representation of $w$ as a product of $l(w)$ transpositions $s_i$, one can remove $l(w)-l(v)$ of the transpositions to obtain a representation of $v$ as a subword in the same relative order.
    \end{definition}
    
    \begin{example}
      The permutation $w=45132$ can be written as $s_2s_3s_2s_1s_4s_3s_2$.  This contains $s_3s_2s_3=14325$ as a (non-consecutive) subword, and so $14325\le 45132$.
    \end{example}

    \subsection{Intersections and Duality}
    
    Now suppose we wish to answer incidence questions about our flags: which flags satisfy certain linear constraints?  As in the case of the Grassmannian, this boils down to understanding how the Schubert varieties $X_w$ intersect. 
    
    We start with the Duality Theorem for $\Fl_n$.  Following \cite{Fulton}, it will be convenient to define dual Schubert varieties as follows.
    
    \begin{definition}
    	Let $E_\bullet$ be the standard and opposite flags, and for shorthand we let $X_w=X_w(F_\bullet)$ and $$Y_w=X_{w_0\cdot w}(E_\bullet)$$ where $w_0=n(n-1)\cdots 1$ is the longest word.  
    	The set $Y_w$ is often called a \textbf{dual Schubert variety}. 
    \end{definition}  
    
    Notice that $$\dim(Y_w)=\inv(w_0\cdot w)=n(n-1)/2-\inv(w)$$ since if $w'=w_0\cdot w$ then $w'(i)=n+1-w(i)$ for all $i$.
    
    \begin{theorem}[Duality Theorem, V2.]
      If $l(w)= l(v)$, we have $X_w\cap Y_v=\emptyset$ if $w\neq v$ and $|X_w\cap Y_v|=1$ if $w=v$.  Furthermore, if $l(w)<l(v)$ then $X_w\cap Y_v=\emptyset$.
    \end{theorem}
    
    The proof works similarly to the Duality Theorem in the Grassmannian.  In particular, with respect to the standard basis, the dual Schubert variety $Y_w$ is formed by the same permutation matrix of $1$'s as in $X_w$, but with the $0$ entries below and to the \textit{right} of the $1$'s (and $\ast$ entries elsewhere).  For instance, we have 
    
    $$X_{45132}=\left(\begin{array}{ccccc} 
    0 & 1 & \ast &  \ast & \ast  \\ 
    1 & 0 & \ast &  \ast & \ast  \\ 
    0 & 0 & 0 & 0 & 1  \\
    0 & 0 & 1 & \ast & 0 \\
    0 & 0 & 0 & 1 & 0 \end{array}\right),
    \hspace{2cm}
    Y_{45132}=\left(\begin{array}{ccccc} 
    \ast & 1    & 0    & 0    & 0  \\ 
    1    & 0    & 0    & 0    & 0  \\ 
    0    & 0    & \ast & \ast & 1  \\
    0    & 0    & 1    & 0    & 0  \\
    0    & 0    & 0    & 1    & 0 \end{array}\right)
    $$
    and their intersection contains only the permutation matrix determined by $w=45132$.
    
    \subsection{Schubert polynomials and the cohomology ring}
    
    In order to continue our variation on the theme, it would be natural at this point to look for a Pieri rule or a Littlewood-Richardson rule.  But just as the cohomology ring of the Grassmannian and the Schur functions made those rules more natural, we now turn to \textbf{Schubert polynomials} and the cohomology ring $H^\ast(\Fl_n)$ over $\ZZ$.
    
    This ring has a natural interpretation as a quotient of a polynomial ring.   In particular, letting $\sigma_w$ be the cohomology class of $Y_w$, we have $\sigma_w\in H^{2i}(\Fl_n)$ where $i=\inv(w)$.  For the transpositions $s_i$, we have $\sigma_{s_i}\in H^2(\Fl_n)$.  The elements $x_i=\sigma_{s_i}-\sigma_{s_{i+1}}$ for $i\le n-1$ and $x_n=-\sigma_{s_{n-1}}$ gives a set of generators for the cohomology ring, and in fact $$H^\ast(\Fl_n)=\mathbb{Z}[x_1,\ldots,x_n]/(e_1,\ldots,e_n)=:R_n$$ where $e_1,\ldots,e_n$ are the elementary symmetric polynomials in $x_1,\ldots,x_n$.   (See \cite{Fulton} or \cite{BGG}.) 
    
     The ring $R_n$ is known as the \textbf{coinvariant ring} and arises in many geometric and combinatorial contexts.  Often defined over a field $k$ rather than $\mathbb{Z}$, its dimension as a $k$-vector space (or rank as a $\mathbb{Z}$-module) is $n!$.  There are many natural bases for $R_n$ of size $n!$, such as the monomial basis given by $$\{x_1^{a_1}\cdots x_n^{a_n}: a_i\le n-i\text{ for all }i\}$$ (see, for instance \cite{GarsiaProcesi}), the harmonic polynomial basis (see \cite{Bergeron}, Section 8.4) and the Schubert basis described below.  There are also many famous generalizations of the coinvariant ring, such as the Garsia-Procesi modules \cite{GarsiaProcesi} and the diagonal coinvariants (see \cite{Bergeron}, Chapter 10), which are closely tied to the study of Macdonald polynomials in symmetric function theory \cite{Macdonald}.
    
    The Schubert polynomials form a basis of $R_n$ whose product corresponds to the intersection of Schubert varieties.  To define them, we require a divided difference operator.
    
    \begin{definition}
    	For any polynomial $P(x_1,\ldots,x_n)\in \mathbb{Z}[x_1,\ldots,x_n]$, we define $$\partial_i(P)=\frac{P-s_i(P)}{x_i-x_{i+1}}$$ where $s_i(P)=P(x_1,\ldots,x_{i-1},x_{i+1},x_i,x_{i+2},\ldots,x_n)$ is the polynomial formed by switching $x_i$ and $x_{i+1}$ in $P$.
    \end{definition}
    
    We can use these operators to recursively define the Schubert polynomials.
    
    \begin{definition}\label{def:Schubert}
    	We define the Schubert polynomials $\Sch_w$ for $w\in S_n$ by:
    	\begin{itemize}
    		\item $\Sch_{w_0}=x_1^{n-1}x_2^{n-2}\cdots x_{n-2}^2 x_{n-1}$ where $w_0=n(n-1)\cdots 21$ is the longest permutation,
    	    \item If $w\neq w_0$, find a minimal factoriation of the form $w=w_0\cdot s_{i_1}\cdot \cdots s_{i_r}$ is a minimal factorization of its form, that is, a factorization for which $\ell(w_0\cdot s_{i_1}\cdot \cdots s_{i_p})=n-p$ for all $1\le p\le r$.  Then $$\Sch_w=\partial_{i_r}\circ \partial_{i_{r-1}}\circ \cdots \circ \partial_{i_1}(\Sch_{w_0})$$
    	\end{itemize}
    \end{definition}
    
    \begin{remark}\label{rem:braid}
     One can show that the operators $\partial_i$ satisfy the two relations below.
     \begin{itemize}
     	\item \textbf{Commutation Relation:} $\partial_i\partial_j=\partial_j\partial_i$ for any $i,j$ with $|i-j|>1$,
     	\item \textbf{Braid Relation:} $\partial_i\partial_{i+1}\partial_i=\partial_{i+1}\partial_i\partial_{i+1}$ for any $i$.
     \end{itemize}
     Since these (along with $s_i^2=1$) generate all relations satisfied by the reflections $s_i$ (see chapter 3 of \cite{BjornerBrenti}), the construction in Definition \ref{def:Schubert} is independent of the choice of minimal factorization.  Note also that $\partial_i^2=0$, so the requirement of minimal factorizations is necessary in the definition.
    \end{remark}
    
    The Schubert polynomials' image in $R_n$ not only form a basis of these cohomology rings, but the polynomials themselves form a basis of all polynomials in the following sense.  The Schubert polynomials $\Sch_w$ are well-defined for permutations $w\in S_\infty=\bigcup S_m$ for which $w(i)>w(i+1)$ for all $i\ge k$ for some $k$.  For a fixed such $k$, these Schubert polynomials form a basis for $\ZZ[x_1,\ldots,x_k]$.
   
    One special case of the analog of the Pieri rule for Schubert polynomials is known as \textbf{Monk's rule}.
    
    \begin{theorem}[Monk's rule]  
    	We have $$\Sch_{s_i}\cdot \Sch_{w}=\sum \Sch_v$$ where the sum ranges over all permutations $v$ obtained from $w$ by:
    	\begin{itemize}
    		\item Choosing a pair $p,q$ of indices with $p\le i<q$ for which $w(p)<w(q)$ and for any $k$ between $p$ and $q$, $w(k)$ is not between $w(p)$ and $w(q)$,
    		\item Defining $v(p)=w(q)$, $v(q)=w(p)$ and for all other $k$, $v(k)=w(k)$.
    	\end{itemize}
    	Equivalently, the sum is over all $v=w\cdot t$ where $t$ is a transposition $(pq)$ with $p\le i<q$ for which $l(v)=l(w)+1$. 
    \end{theorem}
    
    Interestingly, there is not a known ``Littlewood-Richardson rule'' that generalizes Monk's rule, and this is an important open problem in Schubert calculus. 
    
    \begin{problem}
    	Find a combinatorial interpretation analogous to the Littlewood-Richardson rule for the positive integer coefficients $c^{w}_{u,v}$ in the expansion $$\Sch_u\cdot \Sch_v=\sum c^{w}_{u,v}\Sch_w,$$ and therefore for computing the intersection of Schubert varieties in $\Fl_n$.
    \end{problem}
    
    Similar open problems exist for other \textit{partial flag varieties}, defined in the next sections.

    \subsection{Two Alternative Definitions}
    
    There are two other ways of defining the flag manifold that are somewhat less explicit but more generalizable.  The group $\GL_n=\GL_n(\mathbb{C})$ acts on the set of flags by left multiplication on its ordered basis.  Under this action, the stabilizer of the standard flag $F_\bullet$ is the subgroup $B$ consisting of all invertible upper-triangular matrices.  Notice that $\GL_n$ acts transitively on flags via change-of-basis matrices, and so the stabilizer of any arbitrary flag is simply a conjugation $gBg^{-1}$ of $B$.  We can therefore define the flag variety as the set of cosets in the quotient $\GL_n/B$, and define its variety structure accordingly.  
    
    Alternatively, we can associate to each coset $gB$ in $\GL_n/B$ the subgroup $gBg^{-1}$.  Since $B$ is its own \textit{normalizer} in $G$ ($gBg^{-1}=B$ iff $g\in B$), the cosets in $\GL_n/B$ are in one-to-one correspondence with subgroups conjugate to $B$.  We can therefore define the flag variety as the set $\mathcal{B}$ of all subgroups conjugate to $B$.
    
   \subsection{Generalized flag varieties}
   
   The notion of a ``flag variety'' can be extended in an algebraic way starting from the definition as $\GL_n/B$, to quotients of other matrix groups $G$ by certain subgroups $B$ called \textbf{Borel subgroups}.   The subgroup $B$ of invertible upper-triangular matrices is an example of a Borel subgroup of $\GL_n$, that is, a \textbf{maximal connected solvable subgroup}.   It is \textit{connected} because it is the product of the torus $(\mathbb{C}^\ast)^n$ and $\binom{n}{2}$ copies of $\mathbb{C}$.  We can also show that it is \textit{solvable}, meaning that its derived series of commutators 
   \begin{eqnarray*}
   B_0 & := & B, \\ 
   B_1 & := & [B_0,B_0], \\ 
   B_2 & := &[B_1,B_1], \\ 
   & \vdots &
   \end{eqnarray*} 
   terminates.  Indeed, $[B,B]$ is the set of all matrices of the form $bcb^{-1}c^{-1}$ for $b$ and $c$ in $B$.  Writing $b=(d_1+n_1)$ and $c=(d_1+n_2)$ where $d_1$ and $d_2$ are diagonal matrices and $n_1$ and $n_2$ strictly upper-triangular, it is not hard to show that $bcb^{-1}c^{-1}$ has all $1$'s on the diagonal.  By a similar argument, one can show that the elements of $B_2$ have $1$'s on the diagonal and $0$'s on the off-diagonal, and $B_3$ has two off-diagonal rows of $0$'s, and so on.  Thus the derived series is eventually the trivial group.
   
   In fact, a well-known theorem of Lie and Kolchin \cite{Kolchin} states that \textit{all} solvable subgroups of $\GL_n$ consist of upper triangular matrices in some basis.  This implies that $B$ is maximal as well among solvable subgroups.  Therefore $B$ is a Borel subgroup.
   
   The Lie-Kolchin theorem also implies that all the Borel subgroups in $\GL_n$ are of the form $gBg^{-1}$ (and all such groups are Borel subgroups).  That is, all Borel subgroups are conjugate.   It turns out that this is true for any \textbf{semisimple linear algebraic group} $G$, that is, a matrix group defined by polynomial equations in the matrix entries, such that $G$ has no nontrivial smooth connected solvable normal subgroups.  
   
   Additionally, any Borel subgroup in a semisimple linear algebraic group $G$ is its own normalizer.  By an argument identical to that in the previous section, it follows that the groups $G/B$ are independent of the choice of Borel subgroup $B$ (up to isomorphism) and are also isomorphic to the set $\mathcal{B}$ of all Borel subgroups of $G$ as well.  Therefore we can think of $\mathcal{B}$ as an algebraic variety by inheriting the structure from $G/B$ for any Borel subgroup $B$.
   
   Finally, we can define a generalized flag variety as follows. 
   
   \begin{definition}
   	The \textbf{flag variety} of a semisimple linear algebraic group $G$ to be $G/B$ where $B$ is a Borel subgroup.
   \end{definition}
   
   Some classical examples of such linear algebraic groups are the special linear group $SL_n$, the special orthogonal group $SO_n$ of orthogonal $n\times n$ matrices, and the symplectic group $SP_{2n}$ of symplectic matrices.  We will explore a related quotient of the special orthogonal group $SO_{2n+1}$ in Section \ref{sec:variation3}.
   
     We now define partial flag varieties, another generalization of the complete flag variety.  Recall that a \textbf{partial flag} is a sequence $F_{i_1}\subset \cdots \subset F_{i_r}$ of subspaces of $\CC^n$ with $\dim(F_{i_j})=i_j$ for all $j$.  Notice that a $k$-dimensional subspace of $\CC^n$ can be thought of as a partial flag consisting of a single subspace $F_k$.
     
     It is not hard to show that all \textbf{partial flag varieties}, the varieties of partial flags of certain degrees, can be defined as a quotient $G/P$ for a \textbf{parabolic subgroup} $P$, namely a closed intermediate subgroup $B\subset P\subset G$.  The Grassmannian $\mathrm{Gr}(n,k)$, then, can be thought of as the quotient of $\mathrm{GL}_n$ by the parabolic subgroup $S=\mathrm{Stab}(V)$ where $V$ is any fixed $k$-dimensional subspace of $\mathbb{C}^n$.   Similarly, we can start with a different algebraic group, say the special orthogonal group $\mathrm{SO}_{2n+1}$, and quotient by parabolic subgroups to get partial flag varieties of other types.  
     
   \subsection{Problems}
   
   \begin{enumerate}
   	\item \textbf{Reflection length equals inversion number:} Show that $l(w)=\inv(w)$ for any $w\in S_n$.
   	
   	\item \textbf{Practice makes perfect:}  Write out all the Schubert polynomials for permutations in $S_3$ and $S_4$.
   	
   	\item \textbf{Braid relations:}  Verify that the operators $\partial_i$ satisfy the braid relations as stated in Remark \ref{rem:braid}.
   	
   	\item \textbf{The product rule for Schubert calculus:}  Prove that $\partial_i(P\cdot Q)=\partial_i(P)\cdot Q+s_i(P)\cdot \partial_i(Q)$ for any two polynomials $P$ and $Q$. 
   	
   	\item \textbf{Divided difference acts on $R_n$:} Use the previous problem to show that the operator $\partial_i$ maps the ideal generated by elementary symmetric polynomials to itself, and hence the operator descends to a map on the quotient $R_n$.
   	
   	\item \textbf{Schubert polynomials as a basis:}  Prove that if $w\in S_\infty$ satisfies $w(i)>w(i+1)$ for all $i\ge k$ then $\Sch_w\in \ZZ[x_1,\ldots,x_k]$.  Show that they form a basis of the polynomial ring as well. 
   \end{enumerate}
  
\pagebreak
  
\section{Variation 3: The orthogonal Grassmannian}\label{sec:variation3}

In the previous section, we saw that we can interpret the Grassmannian as a partial flag variety.  We can generalize this construction to other matrix groups $G$, hence defining Grassmannians in other Lie types.  We will explore one of these Grassmannians as our final variation.

  \begin{definition}
  	The \textbf{orthogonal Grassmannian} $\mathrm{OG}(2n+1,k)$ is the quotient $\mathrm{SO}_{2n+1}/P$ where $P$ is the stabilizer of a fixed \textbf{isotropic} $k$-dimensional subspace $V$.  The term \textit{isotropic} means that $V$ satisfies $\langle v,w\rangle=0$ for all $v,w\in V$ with respect to a chosen symmetric bilinear form $\langle,\rangle$.
  \end{definition}

  The isotropic condition, at first glance, seems very unnatural.  After all, how could a nonzero subspace possibly be orthogonal to itself?  Well, it is first important to note that we are working over $\mathbb{C}$, not $\mathbb{R}$, and the bilinear form is symmetric, not conjugate-symmetric.  In particular, suppose we define the bilinear form to be the usual dot product $$\langle (a_1,\ldots,a_{2n+1}),(b_1,\ldots,b_{2n+1})\rangle=a_1b_1+a_2b_2+\cdots+a_{2n+1}b_{2n+1}$$ in $\mathbb{C}^{2n+1}$.  Then in $\CC^3$, the vector $(3,5i,4)$ is orthogonal to itself: $3\cdot 3+5i\cdot 5i+4\cdot 4=0$.

  While the choice of symmetric bilinear form does not change the fundamental geometry of the orthogonal Grassmannian, one choice in particular makes things easier to work with in practice: the ``reverse dot product'' given by 
  $$\langle (a_1,\ldots,a_{2n+1}),(b_1,\ldots,b_{2n+1})\rangle=\sum_{i=1}^{2n+1} a_ib_{2n+1-i}.$$  In particular, with respect to this symmetric form, the standard complete flag $F_\bullet$ is an \textbf{orthogonal flag}, with $F_i^\perp=F_{2n+1-i}$ for all $i$.  Orthogonal flags are precisely the type of flags that are used to define Schubert varieties in the orthogonal Grassmannian.

  Note that isotropic subspaces are sent to other isotropic subspaces under the action of the orthorgonal group: if $\langle v,w\rangle=0$ then $\langle Av,Aw\rangle=\langle v,w\rangle=0$ for any $A\in \mathrm{SO}_{2n+1}$.  Thus the orthogonal Grassmannian $\mathrm{OG}(2n+1,k)$, which is the quotient $\mathrm{SO}_{2n+1}/\mathrm{Stab}(V)$, can be interpreted as the variety of all $k$-dimensional isotropic subspaces of $\mathbb{C}^{2n+1}$.

  \subsection{Schubert varieties and row reduction in $\mathrm{OG}(2n+1,n)$}

  Just as in the ordinary Grassmannian, there is a Schubert cell decomposition for the orthogonal Grassmannian.  The combinatorics of Schubert varieties is particularly nice in the case of $\mathrm{OG}(2n+1,n)$ in which the orthogonal subspaces are ``half dimension'' $n$.  (See the introduction of \cite{ThomasYong} or the book \cite{Green} for more details.)

  In $\Gr(2n+1,n)$, the Schubert varieties are indexed by partitions $\lambda$ whose Young diagram fit inside the $n\times (n+1)$ ambient rectangle.  Suppose we divide this rectangle into two staircases as shown below using the blue cut, and only consider the partitions $\lambda$ that are symmetric with respect to the reflective map taking the upper staircase to the lower.
  
  \begin{center}
  \includegraphics{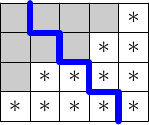}
  \end{center}
  
    We claim that the Schubert varieties of the orthogonal Grassmannian are indexed by the \textbf{shifted partitions} formed by ignoring the lower half of these symmetric partition diagrams.  We define the \textbf{ambient triangle} to be the half of the ambient rectangle above the staircase cut.
    
    \begin{definition}
    	A \textbf{shifted partition} is a strictly-decreasing sequence of positive integers, $\lambda=(\lambda_1 > \ldots > \lambda_k)$.  We write $|\lambda|=\sum \lambda_i$.  The \textbf{shifted Young diagram} of $\lambda$ is the partial grid in which the $i$-th row contains $\lambda_i$ boxes and is shifted to the right $i$ steps.  Below is the shifted Young diagram of the shifted partition $(3,1)$, drawn inside the ambient triangle from the example above.
    	
    	\begin{center}
    		\includegraphics{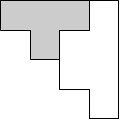}
    	\end{center}
    \end{definition} 
   
    \begin{definition}
    	Let $F_\bullet$ be an orthogonal flag in $\CC^{2n+1}$, and let $\lambda$ be a shifted partition.  Then the \textbf{Schubert variety} $X_\lambda(F_\bullet)$ is defined by
    	$$X_\lambda(F_\bullet)=\{W \in \mathrm{OG}(2n+1,n): \mathrm{dim}(W\cap F_{n+1+i-\overline{\lambda}_i})\ge i\text{ for } i=1,\ldots,n\}$$ where $\overline{\lambda}$ is the ``doubled partition'' formed by reflecting the shifted partition about the staircase.   
    \end{definition}  
    
    In other words, the Schubert varieties consist of the isotropic elements of the ordinary Schubert varieties, giving a natural embedding $\mathrm{OG}(2n+1,n)\to \mathrm{Gr}(2n+1,n)$ that respects the Schubert decompositions: $$X_{\lambda}(F_\bullet) =\Omega_{\overline{\lambda}}(F_\bullet)\cap \mathrm{OG}(2n+1,n).$$

    To get a sense of how this works, consider the example of $\lambda=(3,1)$ and $\overline{\lambda}=(4,3,1)$ shown above, in the case $n=4$.  The Schubert cell $\Omega_{\overline{\lambda}}^{\circ}$ in $\Gr(9,4)$ looks like

\begin{center}
 \includegraphics{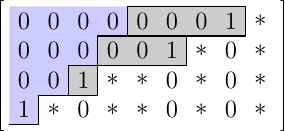}
\end{center} 

    Now, which of these spaces are isotropic?   Suppose we label the starred entries as shown, omitting the $0$ entries:

\begin{center}
\includegraphics{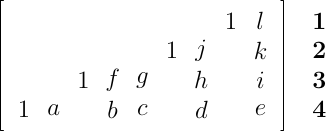}
\end{center}

    We will show that the entries $l,j,k,h,i,e$ are all uniquely determined by the values of the remaining variables $a,b,c,d,f,g$.  Thus there is one isotropic subspace in this cell for each choice of values $a,b,c,d,f,g$, corresponding to the ``lower half'' of the partition diagram we started with, namely
\begin{center}
\includegraphics{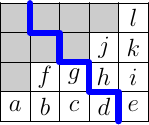}.
\end{center}
To see this, let the rows of the matrix be labeled $\mathbf{1},\mathbf{2},\mathbf{3},\mathbf{4}$ from top to bottom as shown, and suppose its row span is isotropic.  Since row $\mathbf{1}$ and $\mathbf{4}$ are orthogonal with respect to the reverse dot product, we get the relation $$l+a=0,$$ which expresses $l=-a$ in terms of $a$.

Rows $\mathbf{2}$ and $\mathbf{4}$ are also orthogonal, which means that $$b+k=0,$$ so we can similarly eliminate $k$.  From rows $\mathbf{2}$ and $\mathbf{3}$, we obtain $f+j=0$, which expresses $j$ in terms of the lower variables.  We then pair row $\mathbf{3}$ with itself to see that $h+g^2=0$, eliminating $h$, and finally pairing $\mathbf{3}$ with $\mathbf{4}$ we have $i+gc+d=0$, so $i$ is now expressed in terms of lower variables as well.

Moreover, these are the only relations we get from the isotropic condition - any other pairings of rows give the trivial relation $0=0$.  So in this case the Schubert variety restricted to the orthogonal Grassmannian has half the dimension of the original, generated by the possible values for $a,b,c,d,f,g$.

\subsection{General elimination argument}

Why does the elimination process work for any symmetric shape $\lambda$?  Label the steps of the boundary path of $\lambda$ by $1,2,3,\ldots$ from SW to NE in the lower left half, and label them from NE to SW in the upper right half, as shown:
\begin{center}
\includegraphics{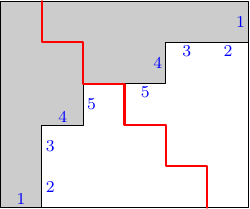}
\end{center}
Then the labels on the vertical steps in the lower left half give the column indices of the $1$'s in the corresponding rows of the matrix.  The labels on the horizontal steps in the upper half, which match these labels by symmetry, give the column indices \textit{from the right} of the corresponding starred columns from right to left.
\begin{center}
\includegraphics{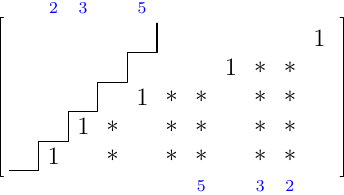}
\end{center}
This means that the $1$'s in the lower left of the matrix correspond to the opposite columns of those containing letters in the upper right half.  It follows that we can use the orthogonality relations to pair a $1$ (which is leftmost in its row) with a column entry in a higher or equal row so as to express that entry in terms of other letters to its lower left.   The $1$ is in a lower or equal row in these pairings precisely for the entries whose corresponding square lies above the staircase cut.  Thus we can always express the upper right variables in terms of the lower left, as in our example above.

\subsection{Shifted tableaux and a Littlewood-Richardson rule}

The beauty of shifted partitions is that so much of the original tableaux combinatorics that goes into ordinary Schubert calculus works almost the same way for shifted tableaux and the orthogonal Grassmannian.  We define these notions rigorously below.

\begin{definition}
	A \textbf{shifted semistandard Young tableau} is a filling of the boxes of a shifted skew shape with entries from the alphabet $\{1'<1<2'<2<3'<3<\cdots \}$ such that the entries are weakly increasing down columns and across rows, and such that primed entries can only repeat in columns, and unprimed only in rows.  
	
	The \textbf{reading word} of such a tableau is the word formed by concatenating the rows from bottom to top.  The \textbf{content} of $T$ is the vector $\mathrm{content}(T) = (n_1, n_2, \ldots)$, where $n_i$ is the total number of $(i)$s and $(i')$s in $T$.  See Figure \ref{fig:shifted} for an example.
\end{definition}  

\begin{figure}
\begin{center}
	\includegraphics{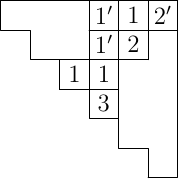}
	\end{center}
	
	\caption{\label{fig:shifted}The tableau above is a shifted semistandard tableau of shape $\lambda/\mu$ where $\lambda=(6,4,2,1)$ and $\mu=(3,2)$, and content $(5,2,1)$.  Its reading word is $3111'21'12'$. } 
	
\end{figure}

In this setting, there are actually two analogs of ``Schur functions'' that arise from these semistandard tableaux.  They are known as the Schur $P$-functions and Schur $Q$-functions.

\begin{definition}
	Let $\lambda/\mu$ be a shifted skew shape.  Define $\mathrm{ShST}_Q(\lambda/\mu)$ to be the set of all shifted semistandard tableaux of shape $\lambda/\mu$.  Define $\mathrm{ShST}_P(\lambda/\mu)$ to be the set of those tableaux in which primes are not allowed on the staircase diagonal.
\end{definition}

\begin{definition}\label{def:SchurQ}
	The \textbf{Schur $Q$-function} $Q_{\lambda/\mu}$ is defined as
	\[Q_{\lambda/\mu}(x_1,x_2,\ldots)=\sum_{T\in \mathrm{ShST}_Q(\lambda/\mu)} x^{\mathrm{wt}(T)}\]
	and the \textbf{Schur $P$-function} $P_{\lambda/\mu}$ is defined as 
	\[P_{\lambda/\mu}(x_1,x_2,\ldots)=\sum_{T\in \mathrm{ShST}_P(\lambda/\mu)} x^{\mathrm{wt}(T)}.\]
\end{definition}

The Schur $Q$-functions, like ordinary Schur functions, are symmetric functions with unique leading terms, spanning a proper subspace of $\Lambda$.  In addition, they have positive product expansions
$$Q_\mu Q_\nu=\sum 2^{\ell(\mu)+\ell(\nu)-\ell(\lambda)}f^{\lambda}_{\mu\nu}Q_\lambda$$ for certain positive integers $f^\lambda_{\mu\nu}$.  It is easy to see that this is equivalent to the rule $$P_\mu P_\nu=\sum f^{\lambda}_{\mu\nu}P_\lambda.$$  Here the coefficients $f^{\lambda}_{\mu\nu}$ are precisely the structure coefficients for the cohomology ring of the orthogonal Grassmannian.  In particular, if we extend them to generalized coefficients by $$P_{\mu^{(1)}}\cdot \cdots \cdot P_{\mu^{(r)}}=\sum f^{\lambda}_{\mu^{(1)}\cdots \mu^{(r)}} P_\lambda,$$ we have the following theorem due to Pragacz \cite{Pragacz}.

\begin{theorem}
	A zero-dimensional intersection $X_{\mu^{(1)}}\cap \cdots \cap X_{\mu^{(r)}}$ has exactly $f^{T}_{\mu^{(1)}\cdots \mu^{(r)}}$ points, where $T$ is the ambient triangle.
\end{theorem}

Stembridge \cite{Stembridge} first found a Littlewood-Richardson-type rule to enumerate these coefficients.  The rule is as follows.

\begin{definition}
	Let $T$ be a semistandard shifted skew tableau with the first $i$ or $i'$ in reading order unprimed, and with reading word $w=w_1\cdots w_n$.  Let $m_i(j)$ be the multiplicity of $i$ among $w_{n-j+1},\ldots,w_n$ (the last $j$ entries) for any $i$ and for any $j\le n$.  Also let $p_{i}(j)$ be the multiplicity of $i'$ among $w_{1},\ldots,w_j$.  Then $T$ is \textbf{Littlewood-Richardson} if and only if
	\begin{itemize}
		\item Whenever $m_i(j)=m_{i+1}(j)$ we have $w_{n-j}\neq i+1,(i+1)'$, and
		\item Whenever $m_i(n)+p_{i}(j)=m_{i+1}(n)+p_{i}(j)$ we have $w_{j+1}\neq i,(i+1)'$.
	\end{itemize}
\end{definition}

Notice that this definition implies that $m_i(j)\ge m_{i+1}(j)$ for all $i$ and $j$, which is similar to the usual Littlewood-Richardson definition for ordinary tableaux.  An alternative rule that only requires reading through the word once (rather than once in each direction, as in the definition of $m_i$ above) is given in \cite{GLP}.

   \subsection{Problems}

    \begin{enumerate}
     	\item\label{converse} Show that, if $\lambda$ is a partition that is \textit{not} symmetric about the staircase cut, the intersection $\Omega_{\lambda}^\circ(F_\bullet)\cap \mathrm{OG}(2n+1,n)$ is empty.
     	\item How many isotropic $3$-planes in $\CC^7$ intersect six given $3$-planes each in at least dimension $1$?
    \end{enumerate}

\pagebreak

\section{Conclusion and further variations}\label{sec:conclusion}

  In this exposition, we have only explored the basics of the cohomology of the Grassmannian, the complete flag variety, and the orthogonal Grassmannian.  There are many other natural directions one might explore from here.  

  First and foremost, we recommend that interested readers next turn to Fulton's book entitled Young Tableaux \cite{Fulton} for more details on the combinatorial aspects of Schubert calculus and symmetric functions, including connections with representation theory.  Other books that are a natural next step from this exposition are those of Manivel \cite{Manivel}, Kumar on Kac-Moody groups and their flag varieties \cite{Kumar}, and Billey-Lakshmibai on smoothness and singular loci of Schubert varieties \cite{BilleyLakshmibai}.

  In some more specialized directions, the flag varieties and Grassmannians in other Lie types (as briefly defined in Section \ref{sec:variation3}) have been studied extensively.  The combinatorics of general Schubert polynomials for other Lie types was developed by Billey and Haiman in \cite{BilleyHaiman} and also by Fomin and Kirillov in type B \cite{FominKirillov}.   Combinatorial methods for minuscule and cominuscule types is presented in \cite{ThomasYong}.

  It is also natural to investigate partial flag varieties between the Grassmannian and $\Fl_n$.  Buch, Kresch, Purbhoo, and Tamvakis established a Littlewood-Richardson rule in the special case of \textit{two-step} flag varieties (consisting of the partial flags having just two subspaces) in \cite{BuchOthers}, and the three-step case was very recently solved by Knutson and Zinn-Justin \cite{KnutsonZinnJustin}.   Coskun provided a potential alternative approach in terms of \textit{Mondrian tableaux}, with a full preliminary answer for partial flag varieties in \cite{Coskun2}, and for the two-row case in \cite{Coskun}. 

  Other variants of cohomology, such as \textit{equivariant cohomology} and \textit{$K$-theory}, have been extensively explored for the Grassmannian and the flag variety as well.  An excellent introduction to equivariant cohomology can be found in \cite{Anderson}, and \cite{Buch} is a foundational paper on the $K$-theory of Grassmannians.  The $K$-theoritic analog of Schubert polynomials are called \textit{Grothendieck polynomials}, first defined by Lascoux and Schutzenberger \cite{LascouxSchutzenberger}.

  Another cohomological variant is \textit{quantum cohomology}, originally arising in string theory and put on mathematical foundations in the 1990's (see \cite{RuanTian}, \cite{KontsevichManin}).  Fomin, Gelfand, and Postnikov \cite{FominGelfandPostnikov} studied a quantum analog of Schubert polynomials and their combinatorics.  Chen studied quantum cohomology on flag manifolds in \cite{Chen}, and the case of equivariant quantum cohomology has been more recently explored by Anderson and Chen in \cite{AndersonChen} and Bertiger, Mili\'{c}evi\'{c}, and Taipale in \cite{BMT}.  In \cite{PechenikYong1} and \cite{PechenikYong2}, Pechenik and Yong prove a conjecture of Knutson and Vakil that gives a rule for equivariant $K$-theory of Grassmannians.  The list goes on; there are many cohomology theories (in fact, infinitely many, in some sense) all of which give slightly different insight into the workings of Grassmannians and flag varieties.

  It is worth noting that Young tableaux are not the only combinatorial objects that can be used to describe these cohomology theories.   Knutson, Tao, and Woodward developed the theory of \textit{puzzles} in \cite{KnutsonTaoWoodward}, another such combinatorial object which often arises in the generalizations listed above.  

  On the geometric side, Vakil \cite{Vakil} discovered a ``geometric Littlewood-Richardson Rule'' that describes an explicit way to degenerate an intersection of Schubert varieties into a union of other Schubert varieties (not just at the level of cohomology).  This, in some sense, more explicitly answers the intersection problems described in Section \ref{sec:intro}.

  Another natural geometric question is the smoothness and singularities of Schubert varieties.  Besides the book by Billey and Lakshmibai mentioned above \cite{BilleyLakshmibai}, this has been studied for the full flag variety by Lakshmibai and Sandya \cite{LakshmibaiSandhya}, in which they found a pattern avoidance criterion on permutations $w$ for which the Schubert variety $X_w$ is smooth.  Related results on smoothness in partial flag varieties and other variants have been studied by Gasharov and Reiner \cite{GasharovReiner}, Ryan \cite{Ryan}, and Wolper \cite{Wolper}.  Abe and Billey \cite{BilleyAbe} summarized much of this work with a number of results on pattern avoidance in Schubert varieties.  

  Real Schubert calculus (involving intersection problems in real $n$-dimensional space $\mathbb{R}^n$) is somewhat more complicated than the complex setting, but there are still many nice results in this area.  For instance, a theorem of Mukhin, Tarasov, and Varchenko in \cite{MTV} states that for a choice of flags that are each maximally tangent at some real point on the rational normal curve, the intersections of the corresponding \textit{complex} Schubert varieties have all real solutions.  An excellent recent overview of this area was written by Sottile in \cite{Sottile}.

  Relatedly, one can study the \textit{positive} real points of the Grassmannian, that is, the subset of the Grassmannian whose Pl\"{u}cker coordinates have all positive (or nonnegative) real values.  Perhaps one of the most exciting recent developments is the connection with scattering amplitudes in quantum physics, leading to the notion of an \textit{amplituhedron} coming from a positive Grassmannian.   An accessible introduction to the main ideas can be found in \cite{BourjailyThomas}, and for more in-depth study, the book \cite{Amplituhedron} by Arkani-Hamed et.\ al.   In \cite{PostnikovSpeyerWilliams}, Postnikov, Speyer, and Williams explore much of the rich combinatorial foundations of the positive Grassmannian.

  Finally, there are also many geometric spaces that have some similarities with the theory of Grassmannians and flag varieties.  \textit{Hessenberg varieties} are a family of subvarieties of the full flag variety determined by stability conditions under a chosen linear transformations (see Tymoczko's thesis \cite{Tymoczko}, for instance).  Lee studied the combinatorics of the affine flag variety in detail in \cite{SeungJin}.  The book \textit{$k$-Schur functions and affine Schubert calculus} by Lam, Lapointe, Morse, Schilling, Shimozono, and Zabrocki \cite{k-Schur} gives an excellent overview of this area, its connections to $k$-Schur functions, and the unresolved conjectures on their combinatorics.  

\pagebreak

\end{document}